\documentclass{article}

\usepackage[preprint,nonatbib]{neurips_2019}

\usepackage[utf8]{inputenc} % allow utf-8 input
\usepackage[T1]{fontenc}    % use 8-bit T1 fonts
\usepackage{url}            % simple URL typesetting
\usepackage{booktabs}       % professional-quality tables
\usepackage{amsfonts}       % blackboard math symbols
\usepackage{nicefrac}       % compact symbols for 1/2, etc.
\usepackage{microtype}      % microtypography
\usepackage{algorithmic,algorithm}
\usepackage{amssymb,amsmath,psfrag,xcolor,natbib,graphicx}
\usepackage[colorlinks,citecolor=bbluegray,linkcolor=ddarkbrown,urlcolor=blue,breaklinks]{hyperref}
\usepackage{bbding}
\definecolor{ddarkbrown}{rgb}{0.5,0.2,0.05} \definecolor{bbluegray}{rgb}{0.05,0,0.5}

\newtheorem{theorem}{Theorem}[section]
\newtheorem{proposition}[theorem]{Proposition}

\newtheorem{lemma}[theorem]{Lemma}
\newtheorem{corollary}[theorem]{Corollary}
\newenvironment{proof}{\textbf{Proof.}}{\QED\bigskip}

\newcommand{\BEAS}{\begin{eqnarray*}}
\newcommand{\EEAS}{\end{eqnarray*}}
\newcommand{\BEA}{\begin{eqnarray}}
\newcommand{\EEA}{\end{eqnarray}}
\newcommand{\BEQ}{\begin{equation}}
\newcommand{\EEQ}{\end{equation}}
\newcommand{\BIT}{\begin{itemize}}
\newcommand{\EIT}{\end{itemize}}
\newcommand{\BNUM}{\begin{enumerate}}
\newcommand{\ENUM}{\end{enumerate}}

\newcommand{\BA}{\begin{array}}
\newcommand{\EA}{\end{array}}

\newcommand{\reals}{{\mathbb R}}

\newcommand{\QED}{~~\rule[-1pt]{6pt}{6pt}}

\newcommand{\argmin}{\mathop{\rm argmin}}

\renewcommand{\algorithmicrequire}{\textbf{Input:}} 
\renewcommand{\algorithmicensure}{\textbf{Output:}}

\let \oldsection \section
\renewcommand{\section}{\vspace{3ex plus 1ex}\oldsection}

\newcommand{\Integer}{\mathbb{N}}

\allowdisplaybreaks

\title{Polyak Steps for Adaptive Fast Gradient Method}
 
\author{%
  Mathieu Barr\'e
  \And
  Alexandre d'Aspremont
}

\begin{document}

\maketitle

\begin{abstract}
Accelerated algorithms for minimizing smooth strongly convex functions usually require knowledge of the strong convexity parameter $\mu$. In the case of an unknown $\mu$, current adaptive techniques are based on restart schemes. When the optimal value $f^*$ is known, these strategies recover the accelerated linear convergence bound without additional grid search. In this paper we propose a new approach that has the same bound without any restart, using an online estimation of strong convexity parameter. We show the robustness of the Fast Gradient Method when using a sequence of upper bounds on $\mu$. We also present a good candidate for this estimate sequence and detail consistent empirical results.
\end{abstract}

\section{Introduction}
We focus on solving a generic optimization problem written
\BEQ
\min f(x)  \triangleq h(x) + \psi(x)
\EEQ
in the variable $x\in\reals^n$, where $h$ is a $L$-smooth, $\mu$-strongly convex function and $\psi(x)$ a convex penalty term. In the deterministic setting, classical convergence bounds show 
\BEQ\label{eq:slow-rate}
\textstyle f(x_k)-f(x_0) \leq \frac{L}{2} \left(1-\frac{\mu}{L}\right)^{k}\|x^*-x_0\|_2
\EEQ
after $k$ iterations of gradient descent with fixed step size, while accelerated proximal gradient descent methods yield iterates satisfying
\BEQ\label{eq:fast-rate}
\textstyle f(x_k)-f(x_0) \leq \frac{L}{2} \left(1-\sqrt{\frac{\mu}{L}}\right)^{k}\|x^*-x_0\|_2
\EEQ
after $k$ iterations, showing a significantly weaker dependence on the problem's condition number $\kappa=L/\mu$ (see \citep{Nest03} for a complete discussion). Similar rates have been obtained in the stochastic setting under the assumption that $h$ is a finite sum. Early work in \citep{Roux12,Shal13,John13,Xiao14,Defa14} produced algorithms with a slow rate roughly matching~\eqref{eq:slow-rate} in its dependence on the condition number. Improved algorithms \citep{Lin14a,Alle16a,Shal14,Lan18} obtain an accelerated rate similar to that in~\eqref{eq:fast-rate}, with \citep{Lan18} in particular showing that these bounds are unimprovable. All these results rely on a strong convexity assumption, with \citep{Arje17} showing that explicit knowledge of the strong convexity constant is required to get the fast rate using simple step size strategies. This remains a key limitation since the strong convexity constant is either unknown or poorly approximated in practice. 

The situation is more favorable in the deterministic setting, with~\citep{Nest13,Lin14b,Ferc16,Roul17,Rene18} showing that the fast rate can be achieved up to a factor $log(\kappa)$, using a restart strategy (the first three references have an extra $1/\mu$ factor in the bound). The results in \citep{Roul17} also show that the $log(\kappa)$ factor can be removed when the value of $f^*$ is known, so that restarted accelerated methods are fully adaptive to strong convexity constant (and other types of growth conditions for that matter). This assumption is often reduced to assuming $f^*=0$ (see e.g. \citep{Asi19} for a more complete discussion), and was used early on to devise better step size strategies for gradient methods, known as {\em Polyak steps} \citep{Poly69,Nedi02}. 

Our objective here is to remove the need for restart. From a practical point of view, while the theoretical bound in \citep{Roul17} is optimal, empirical performance can vary significantly with residual parameter settings. From a theoretical perspective, the need to use a restart scheme highlights the fact that current algorithms and/or convergence analysis fail to capture some key aspects of the problem's regularity properties. Restart schemes are a hack which achieve nearly optimal convergence rates, we seek to find better methods that alleviate the need for these schemes.

We make the following contributions.
\BIT
\item We bound the precision required in estimating the strong convexity parameter $\mu$ to get the fast convergence rate in~\eqref{eq:fast-rate}. In particular, we show that {\em sublinear} convergence in the estimate of $\mu$ is enough to guarantee fast {\em linear} convergence of the iterates.
\item Assuming $f^*$ is known, we detail an efficient strategy to produce local estimates of the strong convexity parameter $\mu$. This estimate has the added benefit of being local, hence better adapts to the geometry of the problem, further speeding up convergence compared to methods given a fixed initial bound on $\mu$.
\item We test our strategy on a variety of learning problems and show that our method often significantly outperforms restart schemes in practice. 
\EIT

\subsection*{Notation}
In what follows, $h$ will denote a $L$-smooth and $\mu$-strongly convex function, $\psi$ a lower-continuous proper convex function. $f(x) := h(x) + \psi(x)$ is then a $\mu$-strongly convex function and $x^*$ will denote the unique minimizer of $f$ on $\reals^n$. Let $f^*=f(x^*)$ be the optimal value of $f$.
$\psi$ will be supposed \textit{simple} enough so that for $\alpha > 0$ the gradient mapping $T_\alpha$
\BEQ
T_\alpha(y) = \underset{x \in \reals^n}{\argmin\;} h(y) + \nabla h(y)^T(x-y) + \frac{\alpha}{2}\|x-y\|^2 +\psi(x)
\EEQ 
can be computed explicitly. Finally the reduced gradient is defined as
\BEQ 
g_\alpha(y) = \alpha(y - T_\alpha(y)).
\EEQ

\section{Nesterov Acceleration of Smooth and Strongly Convex Functions}
In the following we seek to solve the optimization problem
\BEQ
\min f(x) := h(x) + \psi(x)
\EEQ
in the variable $x\in\reals^n$.
\subsection{APG with Known Strong Convexity Parameter}
A classical method for smooth and strongly convex minimization, when the strong convexity parameter is known, is the Accelerated Proximal Gradient (APG) described in Algorithm~\ref{algo:APG}.
\begin{algorithm}[h]
\caption{Accelerated Proximal Gradient}
\label{algo:APG}
\begin{algorithmic}
\STATE \algorithmicrequire\;$x_0 \in \reals^n$, $L$, $\mu$
\STATE $y_{-1} = y_0 = x_0, \beta = \frac{1-\sqrt{\frac{\mu}{L}}}{1+\sqrt{\frac{\mu}{L}}}$.
\FOR{$k \geq 0$}
\STATE $x_{k+1} = y_{k} + \beta(y_k - y_{k-1})$
\STATE $y_{k+1} = T_L(x_{k+1})$
\ENDFOR
\STATE \algorithmicensure\; $y_{k+1}$.
\end{algorithmic}
\end{algorithm}
It can be derived from the generic formulation of the Optimal Gradient Method in \citep[\S2.2.12-13]{Nest18}, using a good choice of estimate sequences and coefficients in order to get only two iterate sequences, $(x_i)_{i\in\Integer}$ and $(y_i)_{i\in\Integer}$, with simple updates. Algorithm~\ref{algo:FGM-const-gen} describes Algorithm~\ref{algo:APG} using an estimate sequence formulation that will prove useful when introducing an estimated strong convexity in the algorithm. A proof of this statement can be found in Appendix~\ref{ap:eq1-2}.
\begin{algorithm}[t]
\caption{APG estimate sequences formulation}
\label{algo:FGM-const-gen}
\begin{algorithmic}
\STATE \algorithmicrequire\;$x_0 \in \reals^n$, $L$, $\mu$
\STATE $A_0 = 1$, $a_0 = 1$, $y_0 = x_0$, $m_0(x) = a_0f^* $. $\kappa = \frac{\mu}{L}$.
\FOR{$k \geq 0$}
\STATE $v_k = \underset{x \in \reals^n}{\argmin}\; m_k(x) + \frac{a_0\mu}{2}\|x-x_0\|^2 $
\STATE $\boxed{a_{k+1} = \frac{\sqrt{\kappa}}{1-\sqrt{\kappa}}A_k}$ 
\STATE $ A_{k+1} = A_k + a_{k+1}$, $\tau_k = \frac{a_{k+1}}{A_{k+1}}$
\STATE $x_{k+1} = \frac{\tau_k}{1+\tau_k}v_k + \frac{1}{1+\tau_k}y_k$
\STATE $y_{k+1} = T_L(x_{k+1})$
\STATE $l_L(x,x_{k+1}) = f(T_L(x_{k+1})) + g_L(x_{k+1})^T(x-x_{k+1}) + \frac{1}{2L}\|g_L(x_{k+1})\|^2 $
\STATE $m_{k+1}(x) = m_k(x) + a_{k+1}\left(l_L(x,x_{k+1})+\frac{\mu}{2}\|x-x_{k+1}\|^2\right)$
\ENDFOR
\STATE \algorithmicensure\; $y_{k+1}$.
\end{algorithmic}
\end{algorithm} \\
We start with the following lemma from \citep{Lin14}, which is an extension of \citep[Th 2.2.13]{Nest18}, and will be used in the analysis.
\begin{lemma}
\label{lem:new-estimseq}
The following inequality holds for $x,y \in \reals^n$.
\BEQ
f(x) \geq f(T_L(y)) + g_L(y)^T(x-y) + \frac{1}{2L}\|g_L(y)\|^2 + \frac{\mu}{2}\|x-y\|^2
\EEQ
\end{lemma}
\begin{proof}
proof in the Appendix~\ref{ap:new-estimseq}
\end{proof}
\begin{corollary}\label{cor:gap-strong}
\BEQ
f(x) - f^* \geq \frac{\mu}{2}\|x-x^*\|^2 , \;\forall x \in \reals^n\EEQ
\end{corollary}
Lemma~\ref{lem:new-estimseq} guarantees that the components of $m_k(x)$ of Algorithm~\ref{algo:FGM-const-gen} are lower bounds on $f(x)$. In particular, we have $m_k(x^*) \leq A_kf^*$. These estimate sequences have also the huge advantage to be strongly convex quadratic functions. Proposition~\ref{prop:proofAPG} now recalls the convergence bound of APG.
\begin{proposition}\label{prop:proofAPG}
After $k$ iterations the output $y_{k}$ of algorithm~\ref{algo:FGM-const-gen} satisfies
\BEQ
f(y_k) - f^* \leq \frac{\left(f(x_0) - f^*\right) + \frac{\mu}{2}\|x_0-x^*\|^2}{A_k}
\EEQ
and
\BEQ
A_k = \left(1-\sqrt{\frac{\mu}{L}}\right)^{-k}\; , \; \forall k \geq 0
\EEQ
\end{proposition}
\begin{proof}
A complete proof using estimate sequence is given in Appendix~\ref{ap:proofAPG}.
\end{proof}\\
This result shows a linear convergence rate in $\left(1-\sqrt{\frac{\mu}{L}}\right)^k$. A linesearch on the smoothness parameter $L$ can be added to the algorithm without losing the convergence bound \cite{Lin14}. In Algorithms~\ref{algo:APG} and~\ref{algo:FGM-const-gen} the strong convexity parameter is given as an input, and is typically hard to estimate. When a misspecified $\hat{\mu} \neq \mu$ is given, two cases are to be distinguished. In the case where we have a lower bound on $\mu$, the proof of Proposition~\ref{prop:proofAPG} still applies because $\mu$ is only used in lower bounds. Linear convergence is preserved and the rate of convergence becomes $(1-\sqrt{\frac{\hat{\mu}}{L}})^k$. When $\hat{\mu}$ is only an upper bound on $\mu$, the previous results only guarantee that the iterates of APG will not blow up (cf. see for instance \citep[Lemma 10]{Lin14}). In what follows we present robustness result on APG, when using an upper bounding sequence that converges to $\mu$ at a sublinear rate.

\subsection{APG with Estimates of Strong Convexity Parameter}
The main result of this section is that for all $k \geq 0$, a sequence $(\mu_i)_{i\in\Integer}$ such that \[0 \leq \mu_i-\mu \leq \frac{C}{(i+1)^2}\] for $i \in [|1,k|]$ so $\mu_i$ converges at a sublinear rate towards $\mu$, allows us to compute $y_k\in\reals^n$ such that \[f(y_k) - f^* \leq C_0\left(1-\sqrt{\frac{\mu}{L}}\right)^k,\] i.e. $f(y_k)$ converges at a linear rate towards $f^*$.

Let $(\mu_i)_{i\in\Integer}$ be a positive real sequence such that $\mu_i \geq \mu, \forall i\geq 0$. Suppose that the $\mu_i$ are available in an online setting, meaning that the $i$-th term can be used at the $i$-th iteration of the algorithm. In the formulation of Algorithm~\ref{algo:FGM-const-gen}, two properties have to be satisfied at each iteration to obtain the convergence bound of Proposition~\ref{prop:proofAPG}.
\BEQ
\left .
    \begin{array}{ll}
        (P_1^k) : & m_k(x^*) \leq A_kf^*  \\
        (P_2^k) : & A_kf(y_k) \leq f(x_0) - f^* + \underset{x \in \reals^n}{\min\;} m_k(x) + \frac{a_0\mu}{2}\|x-x_0\|^2 
    \end{array}
\right \} k \geq 0
\EEQ
The $m_k$ are modified in order to incorporate the strong convexity estimator.
\BEQ
\left \{
    \begin{array}{ll}
        m_0(x) = a_0f^*&\\
    m_{k+1}(x) = m_k(x) + a_{k+1}\left(l_L(x,x_{k+1})+\frac{\mu_{k+1}}{2}\|x-x_{k+1}\|^2 \right), k \geq 0 
    \end{array}
\right .
\EEQ 

Adding these estimate sequences in the APG scheme yields Algorithm~\ref{algo:FGM-const-adapt}.
\begin{algorithm}[h]
\caption{AdaptAPG}
\label{algo:FGM-const-adapt}
\begin{algorithmic}
\STATE \algorithmicrequire\;$x_0 \in \reals^n$, $L$, $(\mu_i)_{i\in\reals^n}$
\STATE $A_0 = 1$, $a_0 = 1$, $y_0 = x_0$, $m_0(x) = a_0f^* $.
\FOR{$k \geq 0$}
\STATE $\kappa_{k} = \frac{\mu_{k+1}}{L}$
\STATE $v_k = \underset{x \in \reals^n}{\argmin}\; m_k(x) + \frac{a_0\mu_0}{2}\|x-x_0\|^2 $
\STATE $\boxed{a_{k+1} = \frac{\sqrt{\kappa_{k}}}{1-\sqrt{\kappa_k}}A_k}$ 
\STATE $ A_{k+1} = A_k + a_{k+1}$, $\tau_k = \frac{a_{k+1}}{A_{k+1}}$
\STATE $x_{k+1} = \frac{\tau_k}{1+\tau_k}v_k + \frac{1}{1+\tau_k}y_k$
\STATE $y_{k+1} = T_L(x_{k+1})$
\STATE $l_L(x,x_{k+1}) = f(T_L(x_{k+1})) + g_L(x_{k+1})^T(x-x_{k+1}) + \frac{1}{2L}\|g_L(x_{k+1})\|^2 $
\STATE $m_{k+1}(x) = m_k(x) + a_{k+1}\left(l_L(x,x_{k+1})+\frac{\mu_{k+1}}{2}\|x-x_{k+1}\|^2\right)$
\ENDFOR
\STATE \algorithmicensure\; $y_{k+1}$.
\end{algorithmic}
\end{algorithm} 
With this choice of recurrence for $a_{k}$, the proximal update for $y_{k+1}$ is preserved. However in this case $x_{k+1}$ can no longer be expressed as a combination of $y_k$ and $y_{k-1}$. In addition, the algorithm keeps the same form of updates as before, ensuring the property $(P_2^k)$ to be preserved at each iteration. However, $(P_1^k)$ relied on the strong convexity lower bounds induced by $\mu$, and these bounds do not hold anymore with  $\mu_i$, introducing additional error terms. Proposition~\ref{prop:APG-estim-bound} below thus gives a preliminary bound on the primal gap depending on the distance between the $\mu_i$ and $\mu$.

\begin{proposition}\label{prop:APG-estim-bound}
Given a non increasing sequence of estimate $(\mu_i)_{i\in \Integer}$ such that $\mu_i \geq \mu ,\forall i>0$, the output of Algorithm~\ref{algo:FGM-const-adapt} after $k$ iterations satisfies
\BEQ
f(y_k) - f^* \leq \frac{f(x_0) - f^* + \frac{a_0\mu}{2}\|x_0 - x^*\|^2}{A_k} + \displaystyle\sum_{i=0}^k\frac{a_i}{2A_k}(\mu_i-\mu)\|x_i-x^*\|^2
\EEQ
and
\BEQ
A_k = \displaystyle\prod_{i=1}^k\left(1-\sqrt{\frac{\mu_i}{L}} \right)^{-1}
\EEQ
\end{proposition}
\begin{proof}
The proof of this result is essentially the same as of Proposition~\ref{prop:proofAPG} and is completely detailed in Appendix~\ref{ap:APG-estim-bound}.
\end{proof}\\
Our goal now is to control the right hand side given sufficient conditions on the gaps $\mu_i - \mu$. In the strongly convex case, the behaviour of the distance to the optimum of the second iterate sequence $(x_i)$ can be controlled. The following lemma uses the form of the update in $x_{k+1}$ as a convex combination of $v_k$ and $y_k$ to bound $\|x_k - x_0\|^2$.

\begin{lemma}\label{lem:bound_sec_seq}
Given $(\mu_i)_{i\in\Integer}$ an non increasing upperbounding sequence of $\mu$. $(x_i)_{i\in\Integer}$ is a sequence defined as in Algorithm~\ref{algo:FGM-const-adapt} on $f$ using $(\mu_i)_{i\in\Integer}$.
\BEQ
\|x_{k+1}-x^*\|^2 \leq \frac{2(f(x_0) - f^*) + \frac{a_0\mu}{2}\|x_0-x^*\|^2}{A_k\mu} + \displaystyle\sum_{i=0}^k\frac{a_i}{A_k}\frac{\mu_i - \mu}{\mu}\|x_i-x^*\|^2 \;,\; \forall k \geq 0
\EEQ
\end{lemma}
\begin{proof}
See Appendix~\ref{ap:bound_sec_seq}
\end{proof}\\
The recurrence equation that defines the $a_k$ allows for a simple bound on the ratio $\frac{a_{k+1}}{A_k}$.
\begin{lemma}\label{lem:bound-ak}
For $(a_i)$ and $(A_i)$ defined as in Algorithm~\ref{algo:FGM-const-adapt}
\BEQ
\frac{a_{k+1}}{A_k} = \frac{\sqrt{\frac{\mu_{k+1}}{L}}}{1-\sqrt{\frac{\mu_{k+1}}{L}}} \leq C_1 = \frac{\sqrt{\frac{\mu_{0}}{L}}}{1-\sqrt{\frac{\mu_{0}}{L}}} 
\EEQ
\end{lemma}
\begin{proof}
$(\mu_i)$ non increasing.
\end{proof}\\
In the next Lemma, we show that when $\mu_i$ converges to $\mu$ at a summable rate, then $\|x_k -x^*\|^2$ converges to 0 with the same speed as $f(y_k) - f^*$. 
\begin{lemma}\label{lem:summable-error}
Given a non increasing sequence $(\mu_i)_{i\in\Integer}$ satisfying
\BEQ
0 \leq \mu_k - \mu \leq \frac{C}{(k+1)^2} \;,\; \forall k \geq 1
\EEQ
with $C \leq \frac{\mu}{3C_1}$ with $C_1$ defined as in Lemma~\ref{lem:bound-ak}. Then for $(a_i)$ and $(x_i)$ defined as in Algorithm~\ref{algo:FGM-const-adapt}
\BEQ
a_k(\mu_k -\mu)\|x_k-x^*\|^2 \leq \frac{C_0}{(k+1)^2}\;,\; \forall k \geq 0
\EEQ
with $C_0 = \max(a_0(\mu_0 - \mu)\|x_0-x^*\|^2,2(f(x_0)-f^*) + a_0\mu\|x_0-x^*\|^2)$
\end{lemma}
\begin{proof}
The proof of this statement can be found in Appendix~\ref{ap:summable-error}.
\end{proof}\\
Now we can prove our main result on robustness of the fast gradient method using upper estimates of the strong convexity parameter.
\begin{proposition}\label{prop:APG-bound}
Given a non increasing sequence $(\mu_i)$ satisfying 
\BEQ
0 \leq \mu_k - \mu \leq \frac{C}{(k+1)^2} \;,\; \forall k \geq 1
\EEQ
with $C \leq \frac{\mu(1-\sqrt{\frac{\mu_0}{L}})}{3\sqrt{\frac{\mu_0}{L}}}$, the output of Algorithm~\ref{algo:FGM-const-adapt} satisfies
\BEQ
f(y_k) - f^* \leq \frac{5C_0}{2A_k} , \forall k \geq 0
\EEQ
where $C_0 = \max((\mu_0 - \mu)\|x_0-x^*\|^2,2(f(x_0)-f^*) + \mu\|x_0-x^*\|^2)$ and
\BEQ
A_k \geq \left(1-\sqrt{\frac{\mu}{L}}\right)^{-k}\; , \; \forall k \geq 1
\EEQ
\end{proposition}
\begin{proof}

Combine Proposition~\ref{prop:APG-estim-bound} and Lemma~\ref{lem:summable-error}. The bound on $A_k$ is true because $\mu_i \leq \mu$. 
\end{proof}\\
This results can be extended in the case where the $\mu_i$ converge at a summable rate to $\mu$. Note also that the constant $C_0$ is bounded by $C_0 = \max((\frac{L}{2} - \mu)\|x_0-x^*\|^2,2(f(x_0)-f^*) + \mu\|x_0-x^*\|^2)$ since the $\mu_0$ will never be taken larger than $\frac{L}{2}$ in our case of interest.

\section{Estimation of Strong Convexity Parameter} 
In this section we propose an estimate of the strong convexity parameter, that can be computed online with the iterations of the algorithm. We do not prove the convergence of our estimate in the general case but we present hints that support its performance. The optimum function value $f^*$ is required to compute these estimates, as for Polyak steps. We set $\mu_0$ to a rough upper bound on $\mu$, for instance $\frac{L}{2}$ is suitable for problems that need to be solved with accelerated methods. Then $\mu_{k+1}$ for $k \geq 0$ is defined as follows
\BEQ\label{eq:mu}
\left .
    \begin{array}{l}
        \hat{\mu}_{k+1} = \frac{\|g_L(y_k)\|^2}{2(f(y_k)-f^*)}\\
        \mu_{k+1} = \underset{i = 0..k+1}{\min\;}\hat{\mu}_i
    \end{array}
\right \} \forall k\geq 0
\EEQ
In the following we keep our study in the case $\psi(x) =0$ and $\hat{\mu}$ becomes
\BEQ
\hat{\mu}_{k+1} = \frac{\|\nabla h(y_k)\|^2}{2(h(y_k) - h^*)},\forall k \geq 0
\EEQ 
Lemma~\ref{lem:nest18} in the Appendix ensures that the $\mu_k$ are lower bounded by the strong convexity $\mu$. The following lemma shows that $\mu_k$ is effectively converging to $\mu$ when the $y_k$ are iterates of a gradient descent on $h$, a strongly convex quadratic. 
\begin{lemma}
Let $h^* \in \reals, x^* \in \reals^n$, $A \in S_n^{++}(\reals)$, and suppose $h(x) = h^* + \frac{1}{2}(x-x^*)^TA(x-x^*)$. Let $y_k$ be the iterates of a gradient descent procedure starting at $y_0$ with constant step $\frac{1}{L}$ where $L$ is the largest eigenvalue of $A$. We get
\BEQ
\frac{\|\nabla h(y_k)\|^2}{2(h(y_k) - h^*)} - \mu \leq \dfrac{\|y_0 - x^*\|^2}{\omega_1^2}(\lambda_2 - \mu)\dfrac{\lambda_2}{\mu}\left( \dfrac{1-\frac{\lambda_2}{L}}{1-\frac{\mu}{L}}\right)^{2k}
\EEQ
where $\mu$ is the smallest eigenvalue of $A$, $\lambda_2$ the second smallest and $\omega_1$ the component of $y_0-x^*$ on the eigenspace associated with $\mu$.
\end{lemma}
\begin{proof}
Decompose the iterates on the eigenvectors of $A$. 
\end{proof}\\
The same kind of convergence with an accelerated rate can be obtain when the $y_{k}$ are the iterates of an APG with a constant momentum $\beta \leq \frac{1-\sqrt{\kappa}}{1+\sqrt{\kappa}}$ on a strongly convex quadratic. The key in these two examples is that the component of $y_k$ associated with the smallest eigenvalue of the hessian of $f$ has the slowest convergence rate. This is the conjugate effect of a gradient step that decreases first the components associated with the highest eigenvalues and of a small extrapolation step that preserves the order of convergence between the different components.   
\section{Numerical Experiments} In this section we present numerical experiments on Algorithm~\ref{algo:FGM-const-adapt}. We also show results of Algorithm~\ref{algo:APGv2}, a very simple modification of APG for which we did not prove robustness but that appears to work very well in practice.
\begin{algorithm}[h]
\caption{APG adapt v2}
\label{algo:APGv2}
\begin{algorithmic}
\STATE \algorithmicrequire\;$x_0 \in \reals^n$, $L$, $f^*$
\STATE $y_{-1} = y_0 = x_0$.
\FOR{$k \geq 0$}
\STATE $\hat{\mu}_k = \frac{\|g_L(y_k)\|^2}{2(f(y_k) - f^*)}$
\STATE $\mu_k = \underset{i = 0..k}{\min}\hat{\mu}_i$
\STATE $\beta_k = \frac{1-\sqrt{\frac{\mu_k}{L}}}{1+\sqrt{\frac{\mu_k}{L}}}$
\STATE $x_{k+1} = y_{k} + \beta_k(y_k - y_{k-1})$
\STATE $y_{k+1} = T_L(x_{k+1})$
\ENDFOR
\STATE \algorithmicensure\; $y_{k+1}$.
\end{algorithmic}
\end{algorithm}

Both Algorithms~\ref{algo:FGM-const-adapt} and \ref{algo:APGv2} compute and use the strong convexity estimates defined in \eqref{eq:mu} during their execution. In order to get the values of $f^*$ in the experiments we run APG for a sufficient amount of time to reach machine precision. We compare our two algorithms (APG adapt) and (APG adapt v2) with Proximal Gradient Descent (PGD), Accelerated Proximal Gradient for smooth functions (APG), Accelerated Proximal Gradient with known strong convexity parameter (APG Optiamal $\mu$) (for square loss and regularized logistic loss) and restarted Accelerated Proxmial Gradient using $f^*$ in a stopping criterion with decay parameter $\gamma$ (APG Restart $\gamma=\cdot$) tuned to give the best result. The restart scheme is described in Appendix~\ref{ap:numeric}. Even though the theoretical complexity bound is optimal, the $\gamma$ tuning step for the restart strategy still has a significant impact on empirical performance, as shown in Figure~\ref{fig:restart-gamma} in the Appendix. In terms of computational cost, our algorithms require one more call to the gradient oracle per iteration than the restarted algorithm but there is no parameter to tune, indeed $\mu_0$ is always chosen as $\frac{L}{2}$ and has no impact in practice.
\begin{figure}[!h]
    \centering
    \includegraphics[scale=0.45]{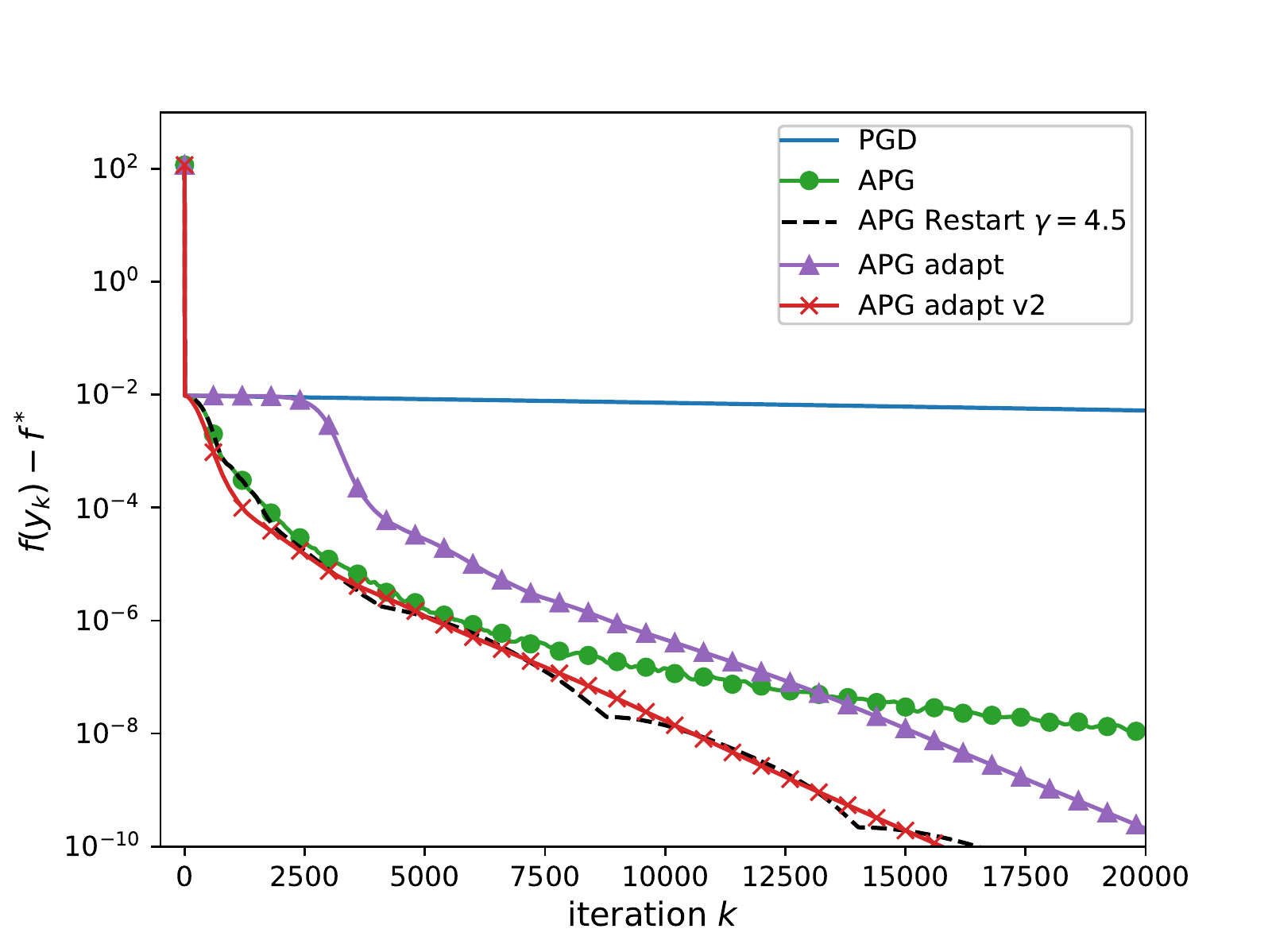}    \caption{Experiments on matrix completion. $f(X) = \sum_{i,j\in\Omega}\|X_{ij}-Y_{ij}\|^2 + \lambda\|X\|_*$ where $Y$ is a random observation matrix in $\reals^{30\times30}$ of rank $5$, $\Omega$ is a subset of $[|1,30|]^2$ of size $200$, $\lambda =0.01$ and $\|\cdot\|_*$ is the nuclear norm.}
    \label{fig:mtx-comp}
\end{figure}

Figure~\ref{fig:mtx-comp} shows the convergence of the primal gap when solving the matrix completion problem on synthetic data using the nuclear norm penalization formulation. Our adaptive algorithms exhibit linear convergence meaning that they successfully estimate the local strong convexity of the problem.\\
Figure~\ref{fig:all} regroups the results of experiments on two real world datasets of different sizes using 4 different classical losses. In all cases, our algorithms perform well and display the fast converging rate. Figure~\ref{fig:sonar} in Appendix~\ref{ap:numeric} shows additional experiments and Figure~\ref{fig:gap} the convergence of our online estimate of the strong convexity parameter during the execution of the algorithm.
%\textbf{Conclusion :} In this paper we proposed an alternative to the restart of accelerated algorithm in order to adapt to strong convexity. We showed a new robustness result of the APG algorithm given upper estimates of the strong convexity parameter. Then we proposed a candidate for an online estimation of this parameter. This choice was motivated theoretically in simple cases and strongly supported by numerical experiments.
\begin{figure}[h]
\def \plotscale {0.43}
\centering
    \begin{tabular}{cc}
    \includegraphics[scale=\plotscale]{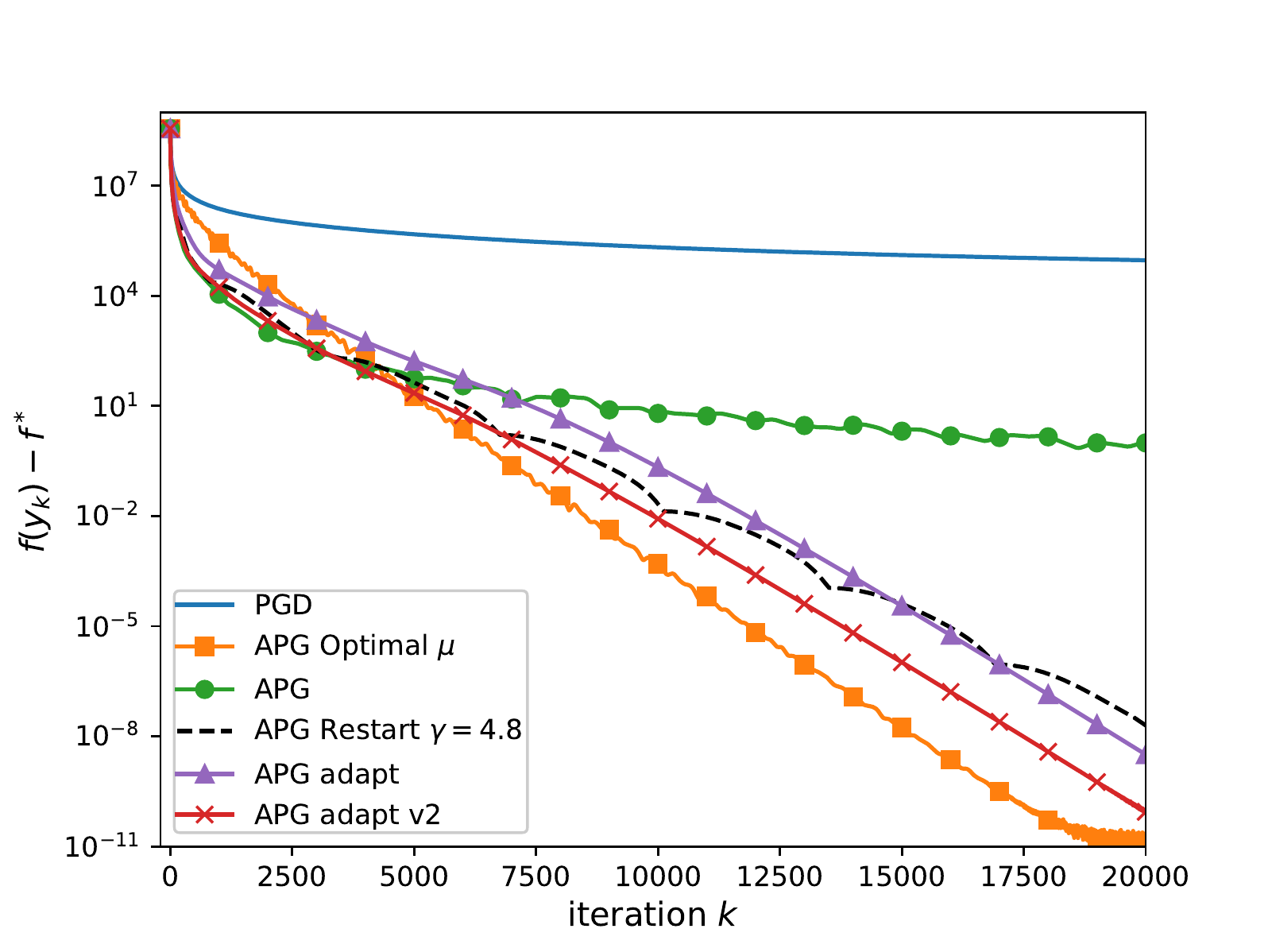}&\includegraphics[scale=\plotscale]{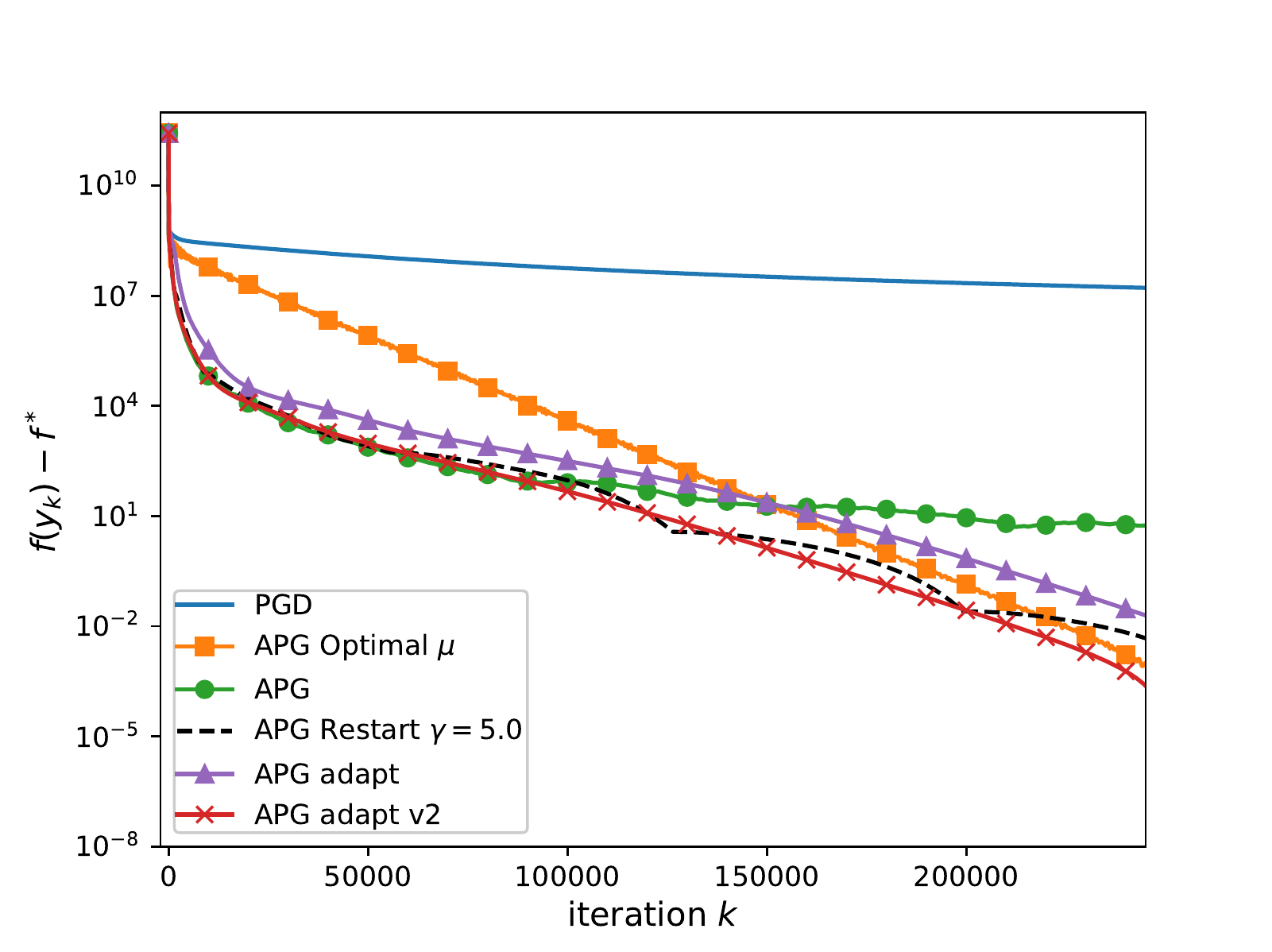}\\    \includegraphics[scale=\plotscale]{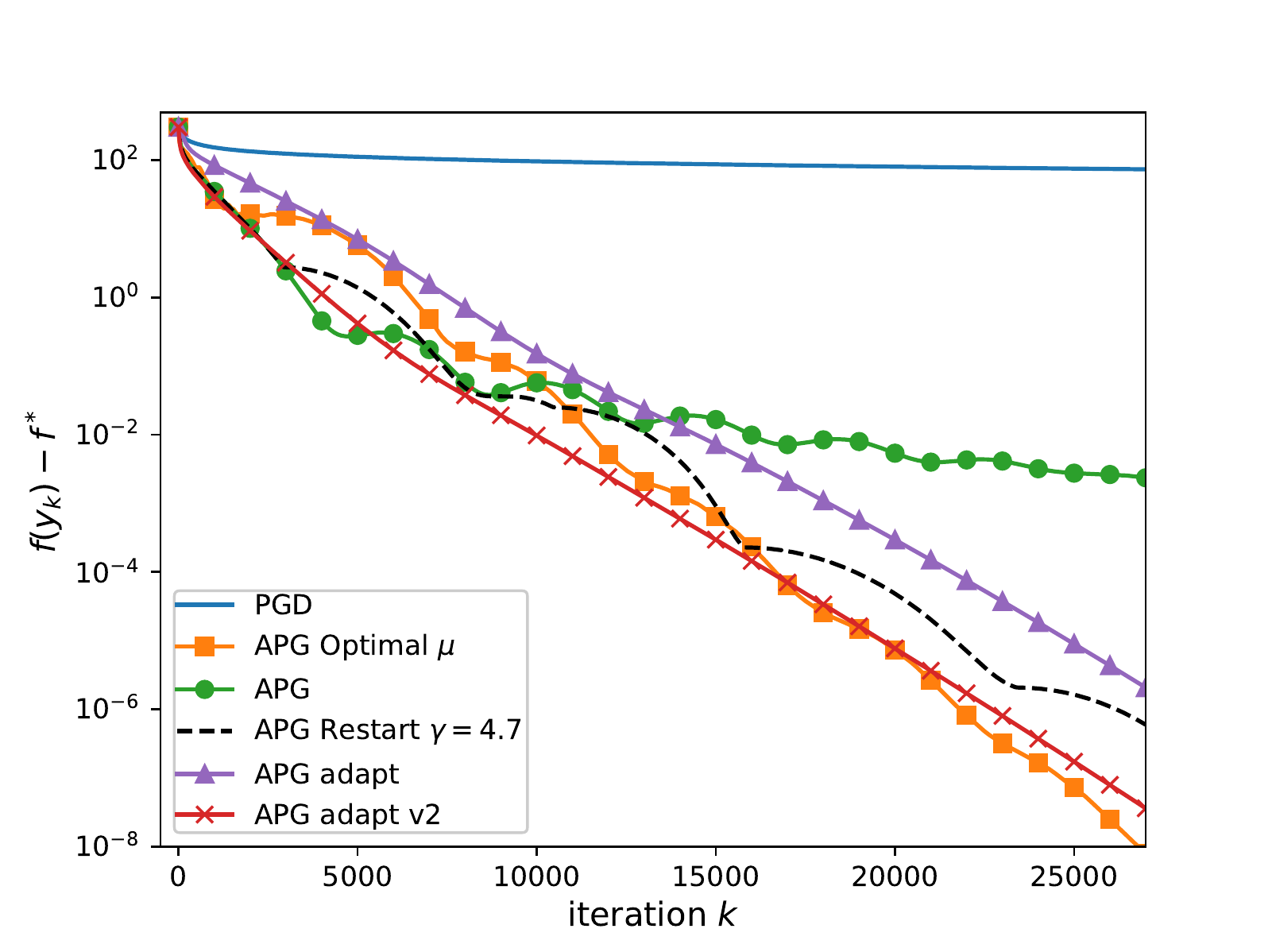}&\includegraphics[scale=\plotscale]{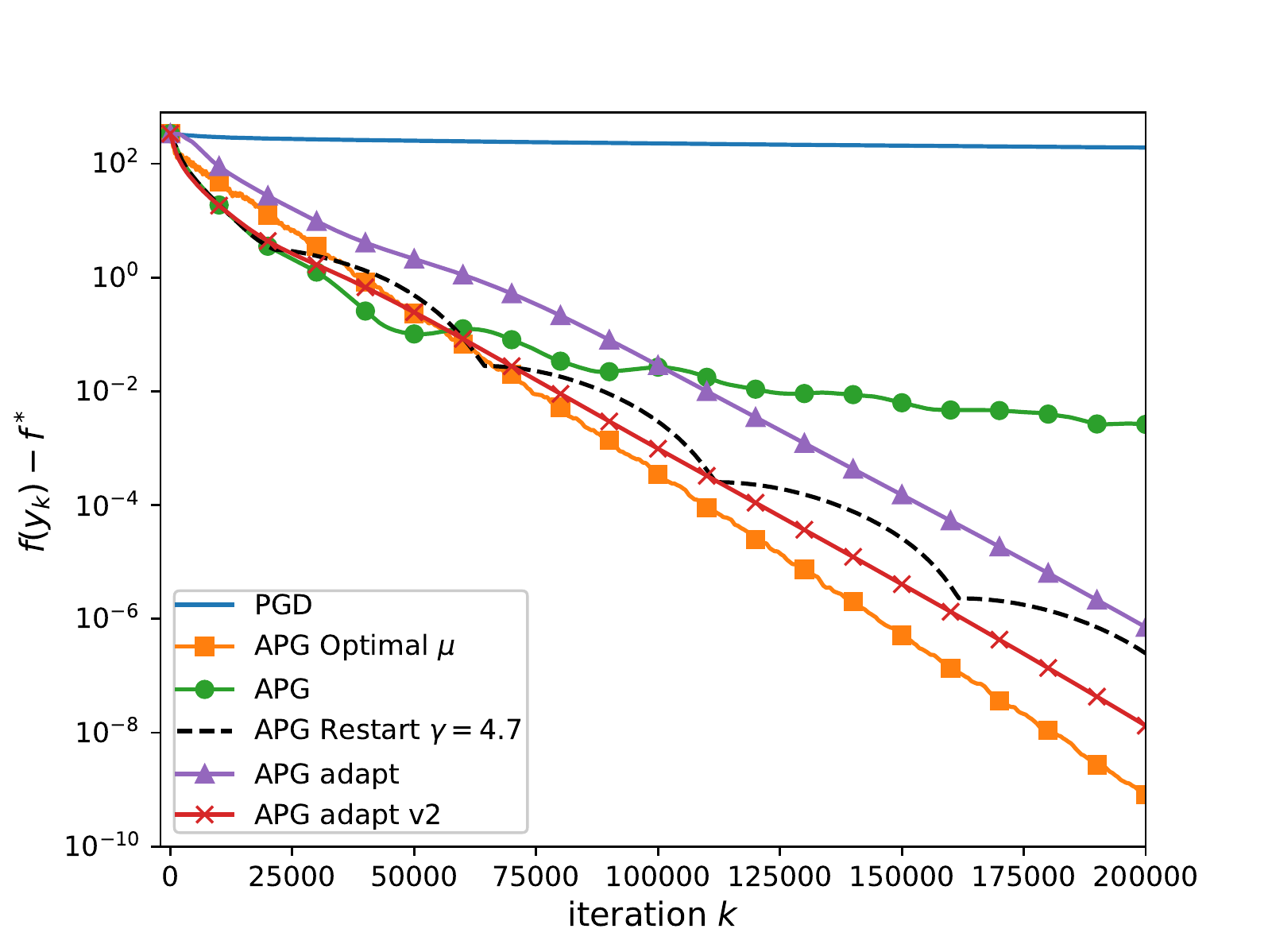}\\
    \includegraphics[scale=\plotscale]{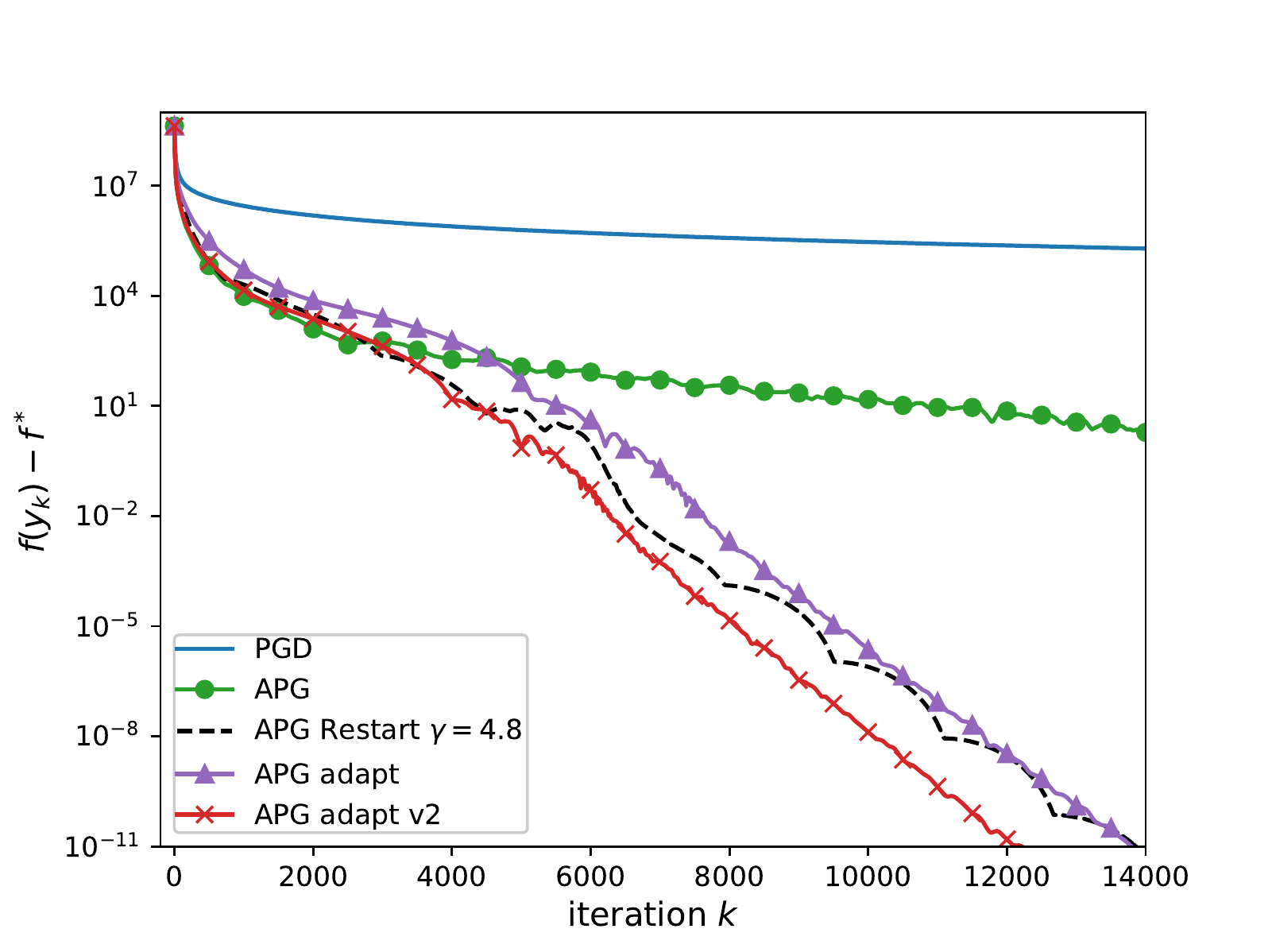}&\includegraphics[scale=\plotscale]{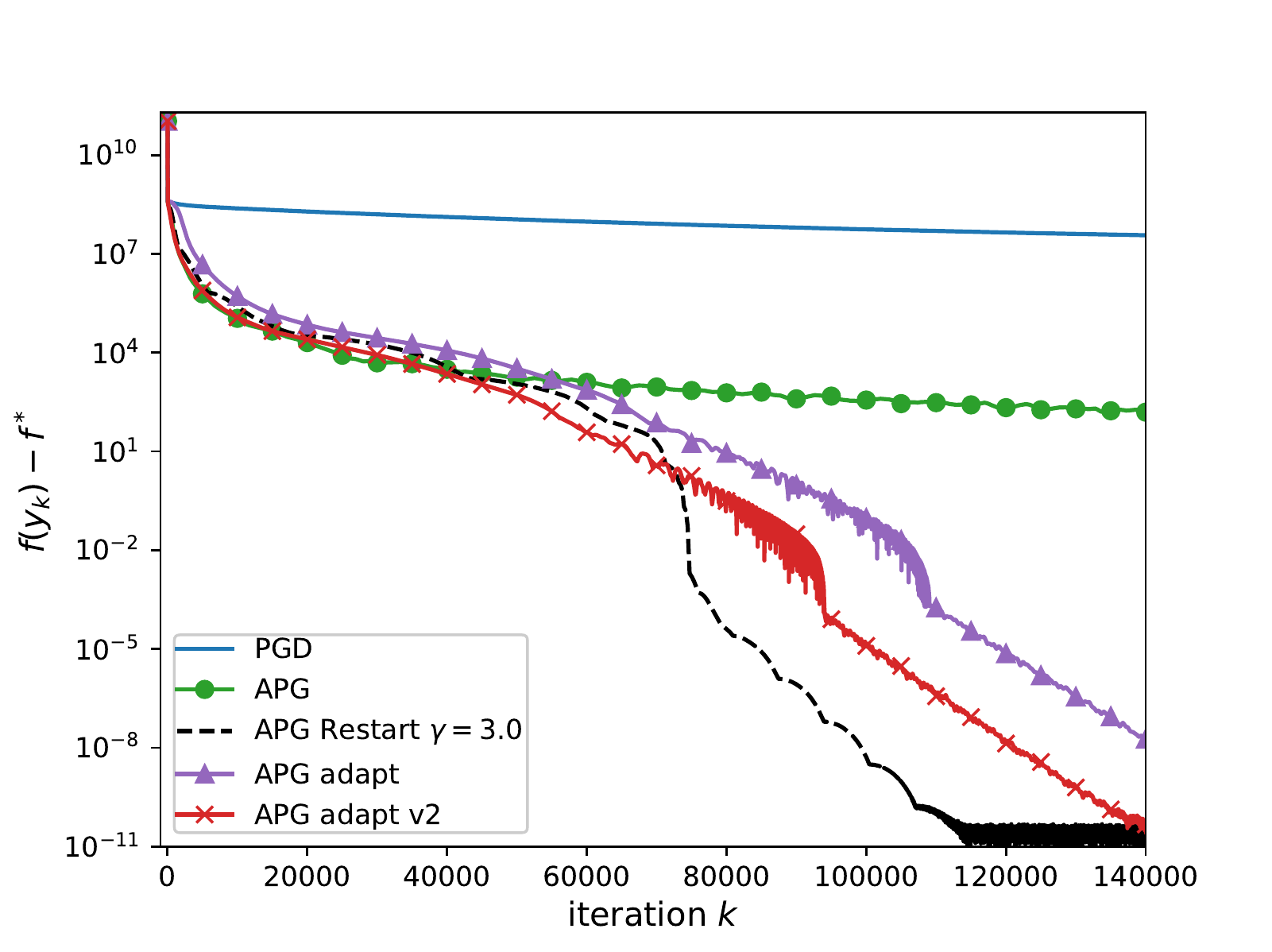}\\
    \includegraphics[scale=\plotscale]{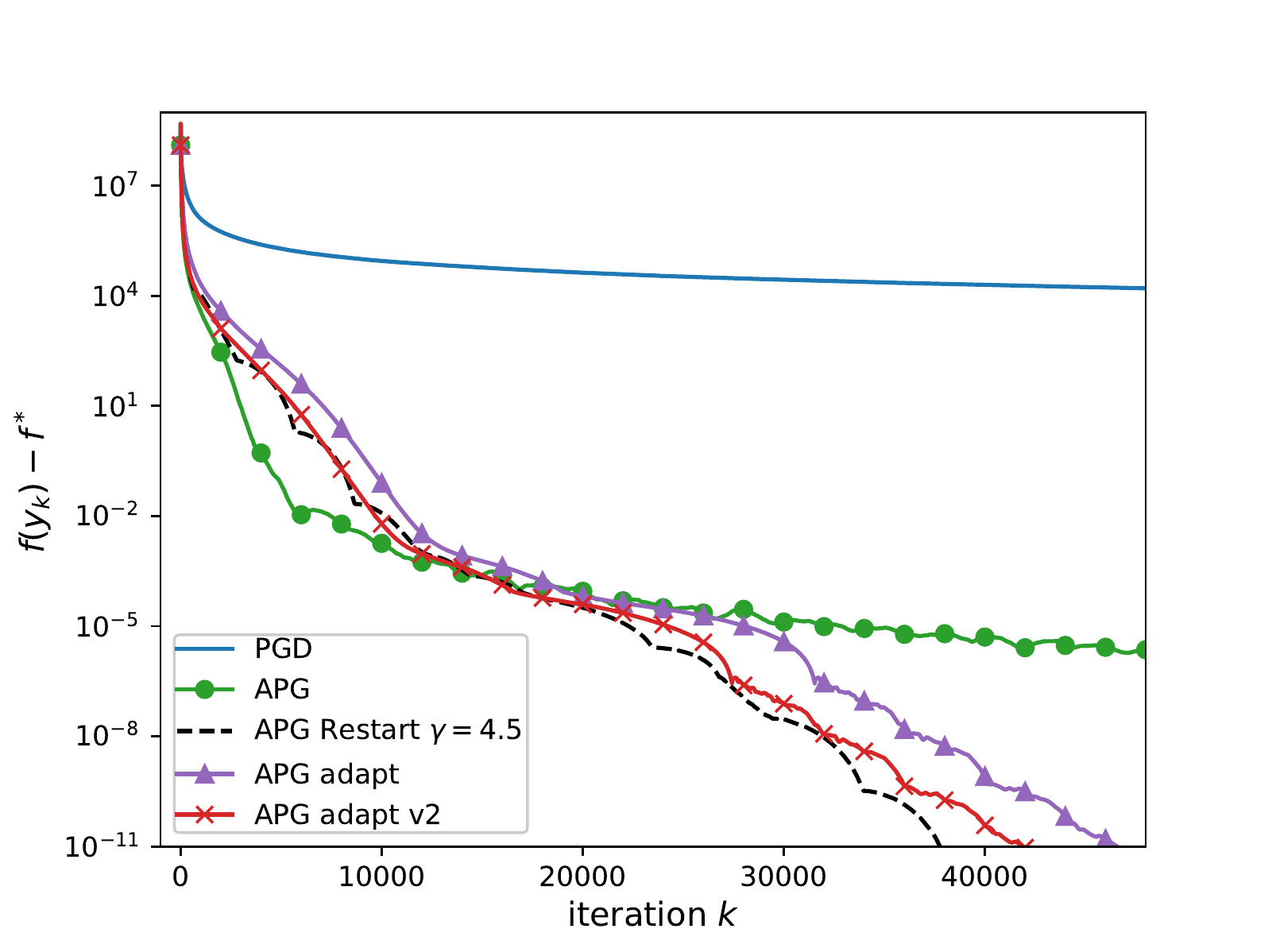}&\includegraphics[scale=\plotscale]{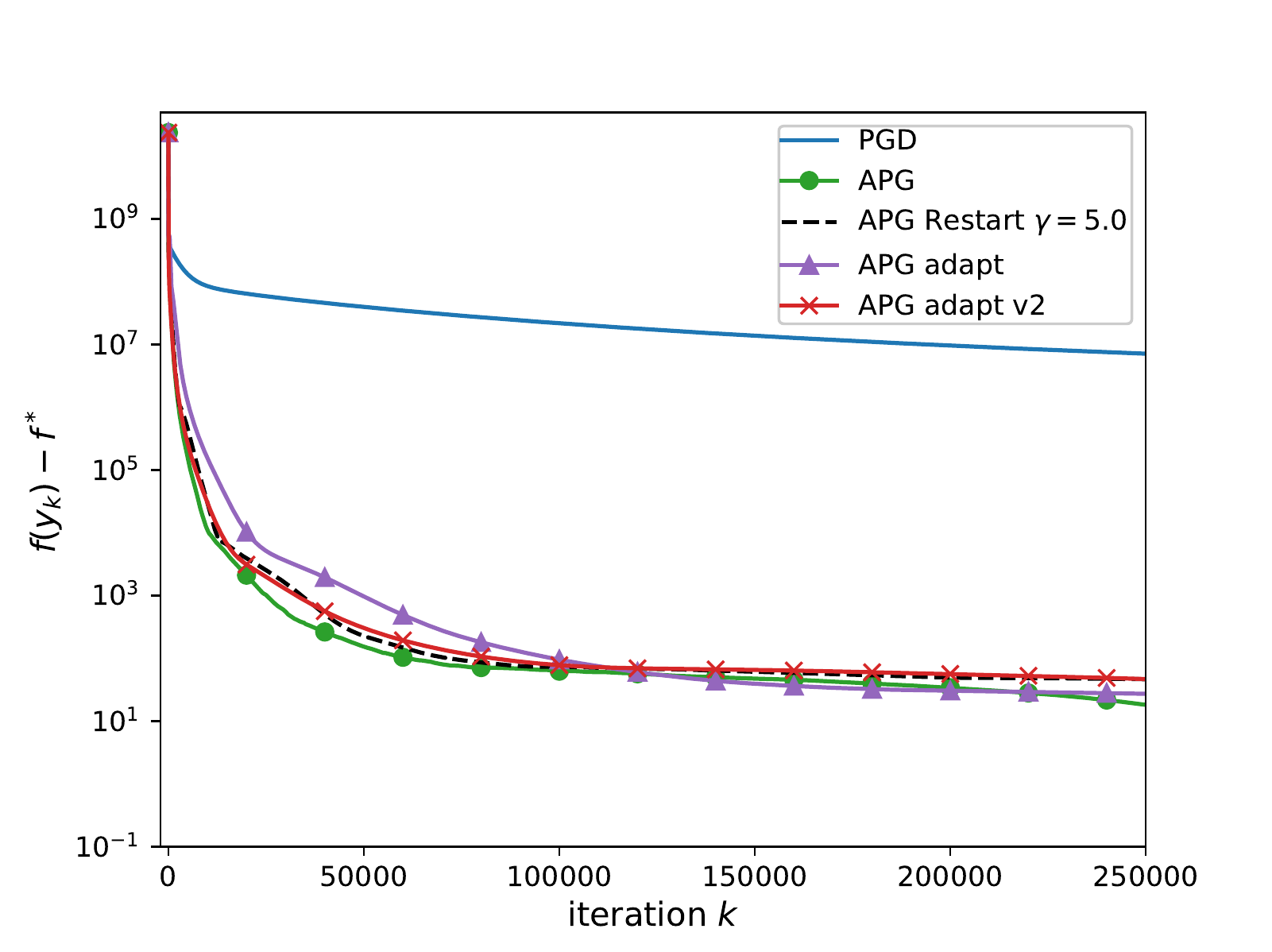}
    \end{tabular}
    \caption{Primal gap versus number of iterations. Each column corresponds to a dataset, Musk (left) and Madelon (right). Each row corresponds to a particular loss, from top to bottom: least square loss, regularized logistic loss, LASSO and dual of regularized SVM. Parameters used in the loss associated with each curve are given in Table~\ref{tab:param} in the Appendix. }
    \label{fig:all}
\end{figure}

\clearpage
\bibliographystyle{plainnat}
\bibliography{MainPerso,mybiblio}

\clearpage
\appendix
{\Large \textbf{Appendix} }
\section{Usefull Lemmas}
\begin{lemma}\label{lem:nest18}
Since $h$ is $L$-smooth and $\mu$-strongly convex, the following bounds hold
\BEQ
h(x) \leq h(y) + \nabla h(y)^T(x-y) +\frac{L}{2}\|x-y\|^2 
\nonumber\EEQ
\BEQ
h(x) \geq h(y) + \nabla h(y)^T(x-y) +\frac{\mu}{2}\|x-y\|^2 
\nonumber\EEQ
$\forall x,y \in \reals^n$.
\end{lemma}
\begin{proof}
\citep[Th 2.1.5, Th 2.1.10]{Nest18}    
\end{proof}
\begin{lemma}\label{ap:eq1-2}
The sequence $(x_i)_{i\in\Integer}$ follows the same updates in Algorithm~\ref{algo:APG} and \ref{algo:FGM-const-gen}.
\end{lemma}
\begin{proof}
Note that $\tau_k = \tau = \sqrt{\kappa}$. Let $\phi_k(x) = m_k(x) + \frac{a_0\mu}{2}\|x-x_0\|^2$ , $\phi$ is a quadratic function. Since $v_{k+1}$ is the $\argmin$ of 
\begin{align*}
    \phi_{k+1}(x) &= \phi_k^* + \frac{A_k\mu}{2}\|x-v_k\|^2 + \frac{a_{k+1}\mu}{2}\|x-x_{k+1}\|^2  \\
    &+a_{k+1}\left(f(T_L(x_{k+1})) + g_L(x_{k+1})^T(x-x_{k+1}) +\frac{1}{2L}\|g_L(x_{k+1})\|^2\right)
\end{align*}
\begin{align*}
v_{k+1} &= (1-\tau)v_k +\tau x_{k+1} - \frac{\tau}{\mu}g_L(x_{k+1})\\
&= (1-\tau)v_k +\tau x_{k+1} - \frac{\tau}{\mu}L(x_{k+1}-y_{k+1})\\
&= (1-\tau)v_k +(\tau +\frac{1}{\tau})x_{k+1} + \frac{1}{\tau}y_{k+1}\\
&= (1-\tau)v_k +(\tau +\frac{1}{\tau})(\frac{\tau}{1+\tau}v_k + \frac{1}{1+\tau}y_k) + \frac{1}{\tau}y{k+1}\\
&=  \frac{\tau -1}{\tau}y_k + \frac{1}{\tau}y_{k+1}
\end{align*}
reinjecting in the expression of $x_{k+2}$,
\begin{align*}
    x_{k+2} &= \frac{\tau}{1+\tau}v_{k+1} + \frac{1}{1+\tau}y_{k+1}\\
    & = \frac{\tau}{1+\tau}(\frac{\tau -1}{\tau}y_k + \frac{1}{\tau}y_{k+1}) + \frac{1}{1+\tau}y_{k+1}\\
    &= y_{k+1} + (\frac{2}{1+\tau} -1) - \frac{1-\tau}{1+\tau}y_k\\
    &= y_{k+1} + \frac{1-\tau}{1+\tau}(y_{k+1}-y_k)
\end{align*}
which is the update of Algorithm~\ref{algo:APG}.
\end{proof}
\section{Proofs of Lemmas and Propositions}
\subsection{Proof of Lemma~\ref{lem:new-estimseq}}\label{ap:new-estimseq}
The optimality condition of $T_L(y)$ can be written $\nabla h(y) - g_L(y) + \xi_L(y) = 0$ with $\xi_L(y)\in\partial\psi(T_L(y))$. By strong convexity of $f$ we have 
\begin{align*}
    f(x) - \frac{\mu}{2}\|x-y\|^2 &\geq h(y) + \nabla h(y)^T(x-y) +\psi(x)\\
    &= h(y) + \nabla h(y)^T(x- T_L(y)) + \nabla h(y)^T(T_L(y) - y) + \psi(x)\\
    &\geq h(T_L(y)) -\frac{L}{2}\|y-T_L(y)\|^2 + \nabla h(y)^T(x- T_L(y)) + \psi(x)\\
    &= h(T_L(y)) -\frac{1}{2L}\|g_L(y)\|^2 + g_L(y)^T(x- T_L(y))\\
    &- \xi_L(y)^T(x- T_L(y)) +\psi(x)\\
    &\geq h(T_L(y)) -\frac{1}{2L}\|g_L(y)\|^2 + g_L(y)^T(x- T_L(y)) +\psi(T_L(y))\\
    &= f(T_L(y)) -\frac{1}{2L}\|g_L(y)\|^2 + g_L(y)^T(x - y) +g_L(y)^T(y - T_L(y))\\
    &= f(T_L(y)) +\frac{1}{2L}\|g_L(y)\|^2 + g_L(y)^T(x- y)
\end{align*}
\subsection{Proof of Proposition~\ref{prop:proofAPG}}\label{ap:proofAPG}
Recall that with this update of $a_k$ we have $\tau_k = \sqrt{\kappa}, \forall k \geq 0$.\\
We have $m_0(x) = a_0f^*$ and Lemma~\ref{lem:new-estimseq} implies $m_k(x) \leq a_0f^* + (A_k-a_0)f^*$. This leads to the useful bound
\BEQ
\underset{x\in\reals^n}{\min }m_k(x) + \frac{a_0\mu}{2}\|x-x_0\|^2 \leq A_kf^* + \frac{a_0\mu}{2}\|x^*-x_0\|^2
\EEQ
Then we show by induction that $A_kf(y_k) \leq \left(f(x_0) - f^* \right) + \underset{x\in\reals^n}{\min} m_k(x) + \frac{a_0\mu}{2}\|x-x_0\|^2 $.\\
At rank $k=0$, $a_0 = 1$, $m_0(x) = f^*$ and $y_0 = x_0$ thus $A_0f(y_0) = f(x_0) - f^* + f^*$.\\
Then suppose the property is true at rank $k$. Denote $\phi_{k+1}(x) \ m_{k+1}(x) + \frac{a_0\mu}{2}\|x-x_0\|^2$
\begin{align*}
    \phi_{k+1}(x) &= m_k(x) +\frac{a_0\mu}{2}\|x-x_0\|^2 \\
    &+ a_{k+1}\left( f(T_L(x_{k+1})) + g_L(x_{k+1})^T(x-x_{k+1}) +\frac{1}{2L}\|g_L(x_{k+1})\|^2\right)\\
    &+ a_{k+1}\frac{\mu}{2}\|x-x_{k+1}\|^2\\
    & \text{using the quadratic form of the estimate sequence}\\
    &\text{ and the induction hypothesis }\\
    & \geq  A_kf(y_k)- (f(x_0) - f^*) + \frac{A_k\mu}{2}\|x-v_k\|^2 + \frac{a_{k+1}\mu}{2}\|x-x_{k+1}\|^2\\
    &+ a_{k+1}\left( f(T_L(x_{k+1})) + g_L(x_{k+1})^T(x-x_{k+1}) +\frac{1}{2L}\|g_L(x_{k+1})\|^2\right)\\
    & \text{let $z_{k+1} = (1-\tau_k)v_k + \tau_kx_{k+1}$, by convexity of $\|x - .\|^2$}\\
    & \geq  A_kf(y_k)- (f(x_0) - f^*) + \frac{A_{k+1}\mu}{2}\|x-z_{k+1}\|^2\\
    &+ a_{k+1}\left( f(T_L(x_{k+1})) + g_L(x_{k+1})^T(x-x_{k+1}) +\frac{1}{2L}\|g_L(x_{k+1})\|^2\right)\\
    &\text{one define } \hat{x}_{k+1} = \underset{x \in \reals^n}{\argmin}\; \frac{A_{k+1}\mu}{2}\|x-z_{k+1}\|^2 + a_{k+1}g_L(x_{k+1})^Tx \\
    & \geq  A_kf(y_k)- (f(x_0) - f^*) + \frac{A_{k+1}\mu}{2}\|\hat{x}_{k+1}-z_{k+1}\|^2\\
    &+ a_{k+1}\left( f(T_L(x_{k+1})) + g_L(x_{k+1})^T(\hat{x}_{k+1}-x_{k+1}) +\frac{1}{2L}\|g_L(x_{k+1})\|^2\right)\\
    & =  A_kf(y_k) - (f(x_0) - f^*) + \frac{A_{k+1}\mu}{2}\|\frac{\tau_k}{\mu}g_L(x_{k+1})\|^2\\
    &+ a_{k+1}\left( f(T_L(x_{k+1})) + g_L(x_{k+1})^T(\hat{x}_{k+1}-x_{k+1}) +\frac{1}{2L}\|g_L(x_{k+1})\|^2\right)\\
    & \text{since } \tau_k = \sqrt{\kappa}\\
    & =  A_kf(y_k) - (f(x_0) - f^*)+ \frac{A_{k+1}}{2L}\|g_L(x_{k+1})\|^2\\
    &+ a_{k+1}\left( f(T_L(x_{k+1})) + g_L(x_{k+1})^T(\hat{x}_{k+1}-x_{k+1}) +\frac{1}{2L}\|g_L(x_{k+1})\|^2\right)\\
    & \text{using Lemma~\ref{lem:new-estimseq} with $\mu$ = 0}\\ 
    & \geq  A_k\left(f(T_L(x_{k+1})) + g_L(x_{k+1})^T(y_k - x_{k+1}) +\frac{1}{2L}\|g_L{x_{k+1}}\|^2)\right) \\
    &- (f(x_0) - f^*)+\frac{A_{k+1}}{2L}\|g_L(x_{k+1})\|^2 \\
    &+ a_{k+1}\left( f(T_L(x_{k+1})) + g_L(x_{k+1})^T(\hat{x}_{k+1}-x_{k+1}) +\frac{1}{2L}\|g_L(x_{k+1})\|^2\right)\\
    &= A_{k+1}f(T_L(x_{k+1}))- (f(x_0) - f^*) \\
    &+ \frac{A_{k+1}}{L}\|g_L(x_{k+1})\|^2 + A_{k+1}g_L(x_{k+1})^T\left((1-\tau_k)y_k + \tau_k\hat{x}_{k+1} - x_{k+1} \right)\\
    & \text{recall that } \hat{x}_{k+1} = z_{k+1}-\frac{\tau_k}{\mu}g_L(x_{k+1}) \text{ and  that } \frac{\tau_k^2}{\mu} = \frac{1}{L}\\
    &= A_{k+1}f(T_L(x_{k+1}))- (f(x_0) - f^*) \\
    &+ A_{k+1}g_L(x_{k+1})^T\left((1-\tau_k)y_k + \tau_kz_{k+1} - x_{k+1} \right)
\end{align*}
We conclude by combining the formulae defining $z_{k+1}$ and $x_{k+1}$.\\
\begin{align*}
    (1-\tau_k)y_k + \tau_kz_{k+1} - x_{k+1} &= (1-\tau_k)y_k + \tau_k(1-\tau_k)v_k + (\tau_k^2-1)x_{k+1}\\
    &=  (1-\tau_k)y_k + \tau_k(1-\tau_k)v_k \\
    &+ (\tau_k^2-1)\left(\frac{\tau_k}{1+\tau_k}v_k + \frac{1}{1+\tau_k}y_k \right)\\
    &= 0
\end{align*}
finally since $y_{k+1} = T_L(x_{k+1})$ we get $ (f(x_0) - f^*) + \underset{x\in\reals^n}{\min\;} \phi_{k+1}(x) \geq A_{k+1}f(y_{k+1})$. 
In addition, $A_{k+1} = \frac{1}{1-\sqrt{\kappa}}A_k$ and $a_0 =1$ leads to $A_k = \left(1-\sqrt{\kappa} \right)^{-k}$.
\subsection{Proof of Proposition~\ref{prop:APG-estim-bound}}\label{ap:APG-estim-bound}
We follow the proof of Proposition~\ref{prop:proofAPG}. However here we have a different bound on $m_k(x)$.\\
$m_k(x) \leq a_0f^* +(A_k-a_0)f(x) + \displaystyle\sum_{i=1}^k \frac{a_i}{2}(\mu_i-\mu)\|x_i-x\|^2  , \forall k \geq 0 $. Which leads to 
\begin{align}
    \underset{x \in \reals^n}{\min } m_k(x) + \frac{a_0\mu_0}{2}\|x-x_0\|^2 &\leq A_kf^* + \displaystyle\sum_{i=1}^k \frac{a_i}{2}(\mu_i-\mu)\|x_i-x^*\|^2 + \frac{a_0\mu_0}{2}\|x_0-x^*\|^2\nonumber\\
     &=A_kf^* + \displaystyle\sum_{i=0}^k \frac{a_i}{2}(\mu_i-\mu)\|x_i-x^*\|^2 + \frac{a_0\mu}{2}\|x_0-x^*\|^2 \label{eq:P1_est}
\end{align}
Now we show by induction that $\boxed{A_kf(y_k) \leq f(x_0) - f^* +\underset{x \in \reals^n}{\min } m_k(x) + \frac{a_0\mu_0}{2}\|x-x_0\|^2}$. At rank $k=0$, $A_0 = a_0 = 1, y_0= x_0$ and $m_0(x) = a_0f^*$, so the property is true. Suppose it is true at rank $k$. LEt $\phi_{k+1}(x) = m_{k+1}(x) + \frac{a_0\mu_0}{2}\|x-x_0\|^2$.
\begin{align*}
    \phi_{k+1}(x) &\leq m_k(x) + \frac{a_0\mu_0}{2}\|x -x_0\|^2 \\
    &+a_{k+1}\left( f(T_L(x_{k+1})) + g_L(x_{k+1})^T(x-x_{k+1}) +\frac{1}{2L}\|g_L(x_{k+1})\|^2 \right)\\
    &  + a_{k+1}\frac{\mu_{k+1}}{2}\|x-x_{k+1}\|^2\\
    & \text{using the quadratic form of the estimate sequence and the induction hypothesis }\\
    & \geq  A_kf(y_k)- (f(x_0) - f^*) + \frac{\sum_{i=0}^ka_i\mu_i}{2}\|x-v_k\|^2 + \frac{a_{k+1}\mu_{k+1}}{2}\|x-x_{k+1}\|^2\\
    &+ a_{k+1}\left( f(T_L(x_{k+1})) + g_L(x_{k+1})^T(x-x_{k+1}) +\frac{1}{2L}\|g_L(x_{k+1})\|^2\right)\\
    & \text{ let $z_{k+1} = (1-\tau_k)v_k + \tau_kx_{k+1}$, by convexity of $\|x - .\|^2$ } \\
    &\text{ and since the $(\mu_i)$ are non increasing} \\
    & \geq  A_kf(y_k)- (f(x_0) - f^*) + \frac{A_{k+1}\mu_{k+1}}{2}\|x-z_{k+1}\|^2\\
    &+ a_{k+1}\left( f(T_L(x_{k+1})) + g_L(x_{k+1})^T(x-x_{k+1}) +\frac{1}{2L}\|g_L(x_{k+1})\|^2\right)\\
    &\text{one define } \hat{x}_{k+1} = \underset{x \in \reals^n}{\argmin}\; \frac{A_{k+1}\mu_{k+1}}{2}\|x-z_{k+1}\|^2 + a_{k+1}g_L(x_{k+1})^Tx \\
    & \geq  A_kf(y_k)- (f(x_0) - f^*) + \frac{A_{k+1}\mu_{k+1}}{2}\|\hat{x}_{k+1}-z_{k+1}\|^2\\
    &+ a_{k+1}\left( f(T_L(x_{k+1})) + g_L(x_{k+1})^T(\hat{x}_{k+1}-x_{k+1}) +\frac{1}{2L}\|g_L(x_{k+1})\|^2\right)\\
    & =  A_kf(y_k) - (f(x_0) - f^*) + \frac{A_{k+1}\mu_{k+1}}{2}\|\frac{\tau_k}{\mu_{k+1}}g_L(x_{k+1})\|^2\\
    &+ a_{k+1}\left( f(T_L(x_{k+1})) + g_L(x_{k+1})^T(\hat{x}_{k+1}-x_{k+1}) +\frac{1}{2L}\|g_L(x_{k+1})\|^2\right)\\
    & \text{here } \tau_k = \sqrt{\kappa_k}\\
    & =  A_kf(y_k) - (f(x_0) - f^*)+ \frac{A_{k+1}}{2L}\|g_L(x_{k+1})\|^2\\
    &+ a_{k+1}\left( f(T_L(x_{k+1})) + g_L(x_{k+1})^T(\hat{x}_{k+1}-x_{k+1}) +\frac{1}{2L}\|g_L(x_{k+1})\|^2\right)\\
    & \text{using Lemma~\ref{lem:new-estimseq} with $\mu = 0$}\\ 
    & \geq  A_k\left(f(T_L(x_{k+1})) + g_L(x_{k+1})^T(y_k - x_{k+1}) +\frac{1}{2L}\|g_L{x_{k+1}}\|^2)\right) \\
    &- (f(x_0) - f^*)+\frac{A_{k+1}}{2L}\|g_L(x_{k+1})\|^2 \\
    &+ a_{k+1}\left( f(T_L(x_{k+1})) + g_L(x_{k+1})^T(\hat{x}_{k+1}-x_{k+1}) +\frac{1}{2L}\|g_L(x_{k+1})\|^2\right)\\
    &= A_{k+1}f(T_L(x_{k+1}))- (f(x_0) - f^*) \\
    &+ \frac{A_{k+1}}{L}\|g_L(x_{k+1})\|^2 + A_{k+1}g_L(x_{k+1})^T\left((1-\tau_k)y_k + \tau_k\hat{x}_{k+1} - x_{k+1} \right)\\
    & \text{recall that } \hat{x}_{k+1} = z_{k+1}-\frac{\tau_k}{\mu_{k+1}}g_L(x_{k+1}) \text{ and  that } \frac{\tau_k^2}{\mu_{k+1}} = \frac{1}{L}\\
    &= A_{k+1}f(T_L(x_{k+1}))- (f(x_0) - f^*) \\
    &+ A_{k+1}g_L(x_{k+1})^T\left((1-\tau_k)y_k + \tau_kz_{k+1} - x_{k+1} \right)
\end{align*}
We conclude by combining the formulae defining $z_{k+1}$ and $x_{k+1}$.
\begin{align*}
    (1-\tau_k)y_k + \tau_kz_{k+1} - x_{k+1} &= (1-\tau_k)y_k + \tau_k(1-\tau_k)v_k + (\tau_k^2-1)x_{k+1}\\
    &=  (1-\tau_k)y_k + \tau_k(1-\tau_k)v_k \\
    &+ (\tau_k^2-1)\left(\frac{\tau_k}{1+\tau_k}v_k + \frac{1}{1+\tau_k}y_k \right)\\
    &= 0
\end{align*}
finally since $y_{k+1} = T_L(x_{k+1})$ we get $ (f(x_0) - f^*) + \underset{x\in\reals^n}{\min\;} \phi_{k+1}(x) \geq A_{k+1}f(y_{k+1})$, re-injecting in \eqref{eq:P1_est} gives the right bound. 
In addition, $A_{k+1} = \frac{1}{1-\sqrt{\frac{\mu_{k+1}}{L}}}A_k$ and $a_0 =1$ leads to $A_k = \prod_{i=1}^k\left(1-\sqrt{\frac{\mu_i}{L}}\right)^{-1}$.
\subsection{Proof of Lemma~\ref{lem:bound_sec_seq}}\label{ap:bound_sec_seq}
From the definition of $x_{k+1}$ in Algorithm~\ref{algo:FGM-const-adapt}, $x_{k+1} =\alpha_kv_k + (1-\alpha_k)y_k$ with $\alpha_k =  \frac{\tau_k}{1+\tau_k} \in [0,1]$. By convexity of $\|\cdot - x^*\|^2$ 
\BEQ\label{eq:conv}
\|x_{k+1} - x^*\|^2 \leq \alpha_k\|v_k -x^*\|^2 + (1-\alpha_k)\|y_k -x^*\|^2
\EEQ
We denote $\phi_k(x) = m_k(x) + \frac{a_0\mu_0}{2}\|x-x_0\|^2 $, we have that $v_k = \underset{x\in\reals^n}{\argmin\;} \phi_k(x)$. Note that $\phi_k(x)$ is $\left(\sum_{i=0}^ka_i\mu_i\right)$-strongly convex, which gives
\begin{align*}
    \frac{\sum_{i=0}^k a_i\mu_i}{2}\|v_k - x^*\|^2 &\leq \phi_k(x^*) - \phi^* \\ 
     &\leq \phi_k(x^*) - A_kf(y_k) + (f(x_0)-f^*) \\
    &\leq A_kf^* - A_kf(y_k) +\sum_{i=0}^ka_i\frac{\mu_i -\mu}{2}\|x_i - x^*\|^2 + (f(x_0)-f^*) \\
    &+ \frac{a_0\mu}{2}\|x_0-x^*\|^2 \\
    & \text{ since the $\mu_i$ are upperbounds on $\mu$ and $f^* - f(y_k) \leq 0$}\\
    \frac{A_k\mu}{2}\|v_k - x^*\|^2 &\leq \sum_{i=0}^ka_i\frac{\mu_i -\mu}{2}\|x_i - x^*\|^2 + (f(x_0)-f^*) + \frac{a_0\mu}{2}\|x_0-x^*\|^2 \\
    \|v_k - x^*\|^2 &\leq \sum_{i=0}^k\frac{a_i}{A_k}\frac{\mu_i -\mu}{\mu}\|x_i - x^*\|^2 +\frac{2(f(x_0)-f^*) + a_0\mu\|x_0-x^*\|^2}{\mu A_k}
\end{align*}
We can bound $\|y_k-x^*\|^2$ the same way using Corollary~\ref{cor:gap-strong}
\begin{align*}
    \|y_k-x^*\|^2 &\leq \frac{2}{\mu}(f(y_k)-f^*)\\
    \|y_k-x^*\|^2 &\leq \sum_{i=0}^k\frac{a_i}{A_k}\frac{\mu_i -\mu}{\mu}\|x_i - x^*\|^2 +\frac{2(f(x_0) - f^*) + a_0\mu\|x_0-x^*\|^2}{A_k\mu}
\end{align*}
combining these inequality in \eqref{eq:conv} gives the result.
\subsection{Proof of Lemma~\ref{lem:summable-error}}\label{ap:summable-error}
We prove our result by induction. For $k = 0$ this is true since $C_0 \geq  a_0(\mu_0 - \mu)\|x_0-x^*\|^2$. Now suppose the property is true until a rank $k \geq 0$.\\
By Lemma~\ref{lem:bound_sec_seq}, 
\begin{align*}
    \|x_{k+1}-x^*\|^2 &\leq \frac{C_0}{A_k\mu} + \displaystyle\sum_{i=0}^k\frac{a_i}{A_k}\frac{\mu_i - \mu}{\mu}\|x_i-x^*\|^2\\
    & \text{using the induction hypothesis}\\
    & \leq \frac{C_0}{A_k\mu} + \displaystyle\sum_{i=1}^k\frac{C_0}{(i+1)^2A_k\mu}\\
    &\text{ note that $\sum_{i=1}^\infty\frac{1}{i^2} \leq 2$} \\
    & \leq \frac{3C_0}{A_k\mu}
\end{align*} 
Thus
\begin{align*}
    a_{k+1}(\mu_{k+1}-\mu)\|x_{k+1}-x^*\|^2 &\leq \frac{3C_0}{\mu}\frac{a_{k+1}}{A_k}(\mu_{k+1}-\mu)\\
    &\leq \frac{3C_0}{\mu}C_1(\mu_{k+1}-\mu)\\
    \text{ using Lemma~\ref{lem:bound-ak}}\\
    &\leq \frac{3C_0}{\mu}C_1\frac{C}{(k+2)^2} \text{ since } k+1 \geq 1\\
    &\text{we have by hypothesis } C \leq \frac{\mu}{3C_1}\\
    & \leq \frac{C_0}{(k+2)^2}
\end{align*}
which concludes the proof.
\section{Numerical Experiments}\label{ap:numeric}
\begin{algorithm}[h]
\caption{APG Smooth \cite{Beck09}}
\label{algo:APGsmooth}
\begin{algorithmic}
\STATE \algorithmicrequire\;$x_0 \in \reals^n$, $L$,
\STATE $y_{-1} = y_0 = x_0$, $t_0 = 1$
\FOR{$k \geq 0$}
\STATE $\mu_k = \underset{i = 0..k}{\min}\hat{\mu}_i$
\STATE $t_{k+1} = \frac{1+\sqrt{1+t_k^2}}{2}$
\STATE $\beta_k = \frac{t_k-1}{t_{k+1}}$
\STATE $x_{k+1} = y_{k} + \beta_k(y_k - y_{k-1})$
\STATE $y_{k+1} = T_L(x_{k+1})$
\ENDFOR
\STATE \algorithmicensure\; $y_{k+1}$.
\end{algorithmic}
\end{algorithm}
\begin{algorithm}[h]
\caption{APG Restart with Known $f^*$ \cite{Roul17}}
\label{algo:APGrestart}
\begin{algorithmic}
\STATE \algorithmicrequire\;$x_0 \in \reals^n$, $L$, $f^*$, $\gamma > 0$
\STATE $y_{-1} = y_0 = x_0$, $t_0 = 1$, $\epsilon = f(y_0) - f^*$
\FOR{$k \geq 0$}
\STATE $\mu_k = \underset{i = 0..k}{\min}\hat{\mu}_i$
\STATE $t_{k+1} = \frac{1+\sqrt{1+t_k^2}}{2}$
\STATE $\beta_k = \frac{t_k-1}{t_{k+1}}$
\STATE $x_{k+1} = y_{k} + \beta_k(y_k - y_{k-1})$
\STATE $y_{k+1} = T_L(x_{k+1})$
\IF{$f(y_{k+1}) - f^* \leq \epsilon$}
\STATE $t_{k+1} \leftarrow 1$, $x_{k+1} \leftarrow y_{k+1}$, $y_{k} \leftarrow y_{k+1}$, $\epsilon \leftarrow e^{-\gamma}\epsilon$
\ENDIF
\ENDFOR
\STATE \algorithmicensure\; $y_{k+1}$.
\end{algorithmic}
\end{algorithm}
\begin{figure}[h]
    \centering
    \includegraphics[scale=0.45]{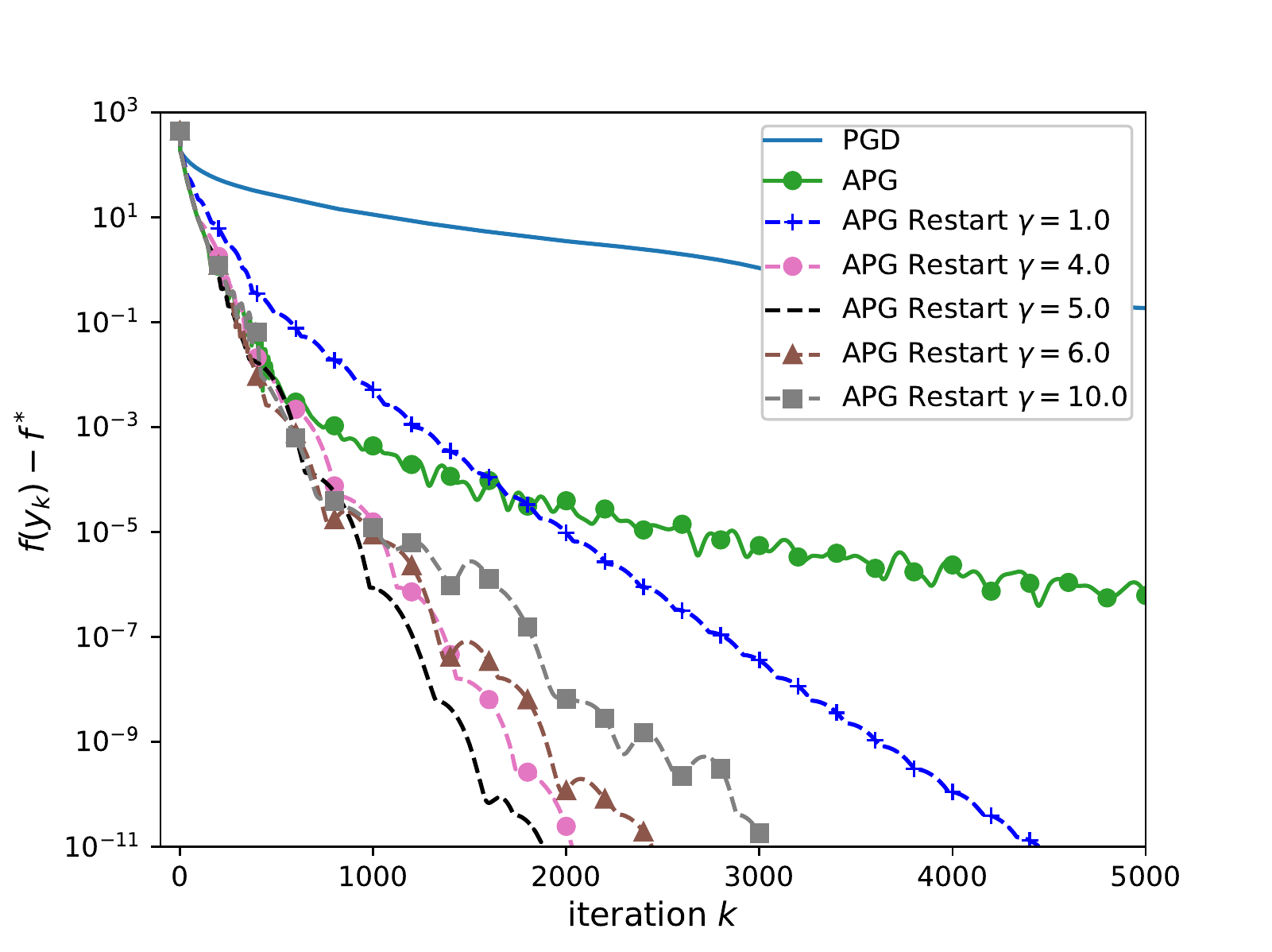}
    \caption{Solving LASSO on Sonar dataset with regularization parameter equal to 1. We observe a large variability with the choice of $\gamma$ in the restarted algorithm.}
    \label{fig:restart-gamma}
\end{figure}
\begin{figure}[h]
    \centering
    \begin{tabular}{cc}
       \includegraphics[scale = 0.43]{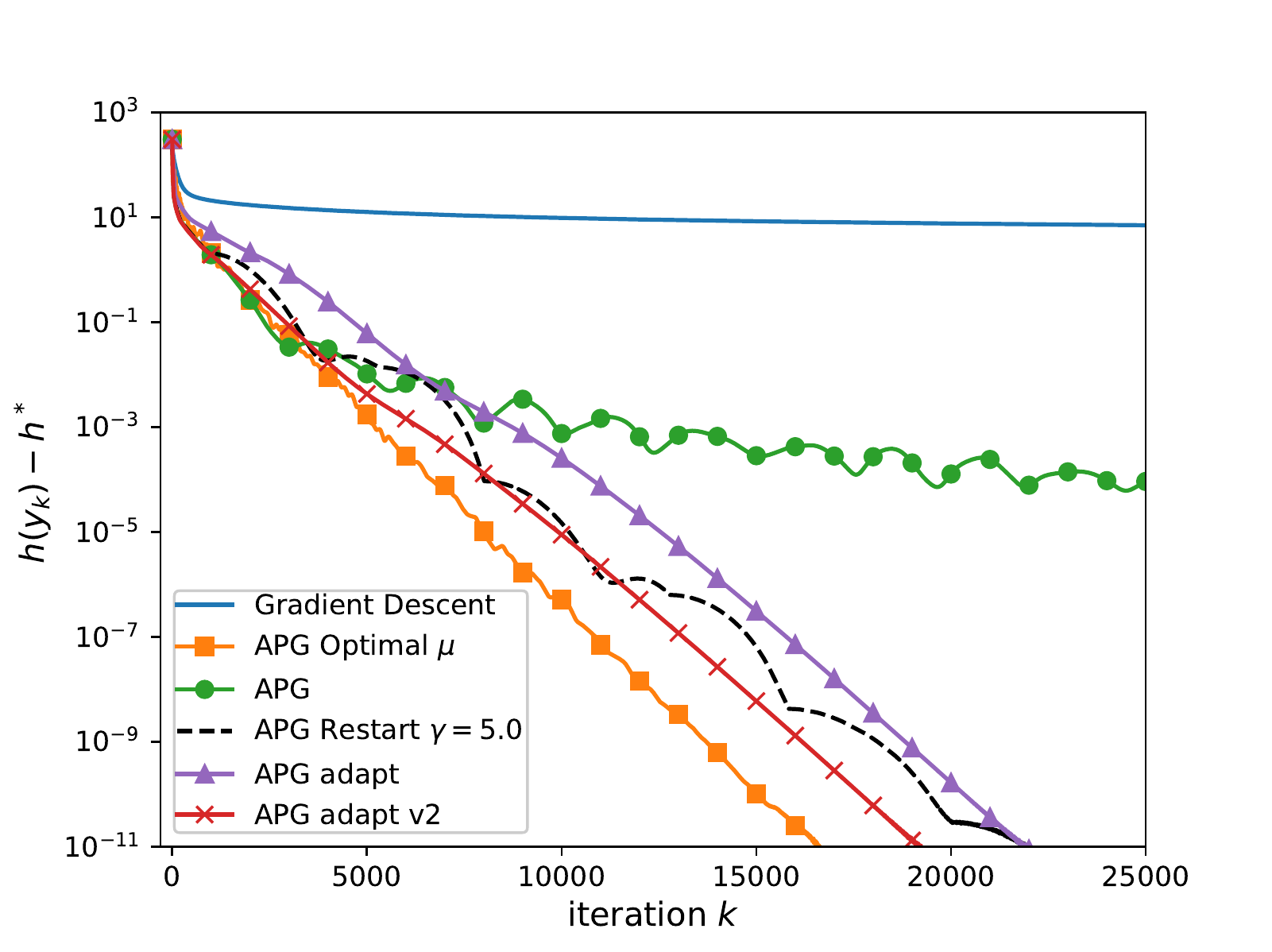}  & \includegraphics[scale = 0.43]{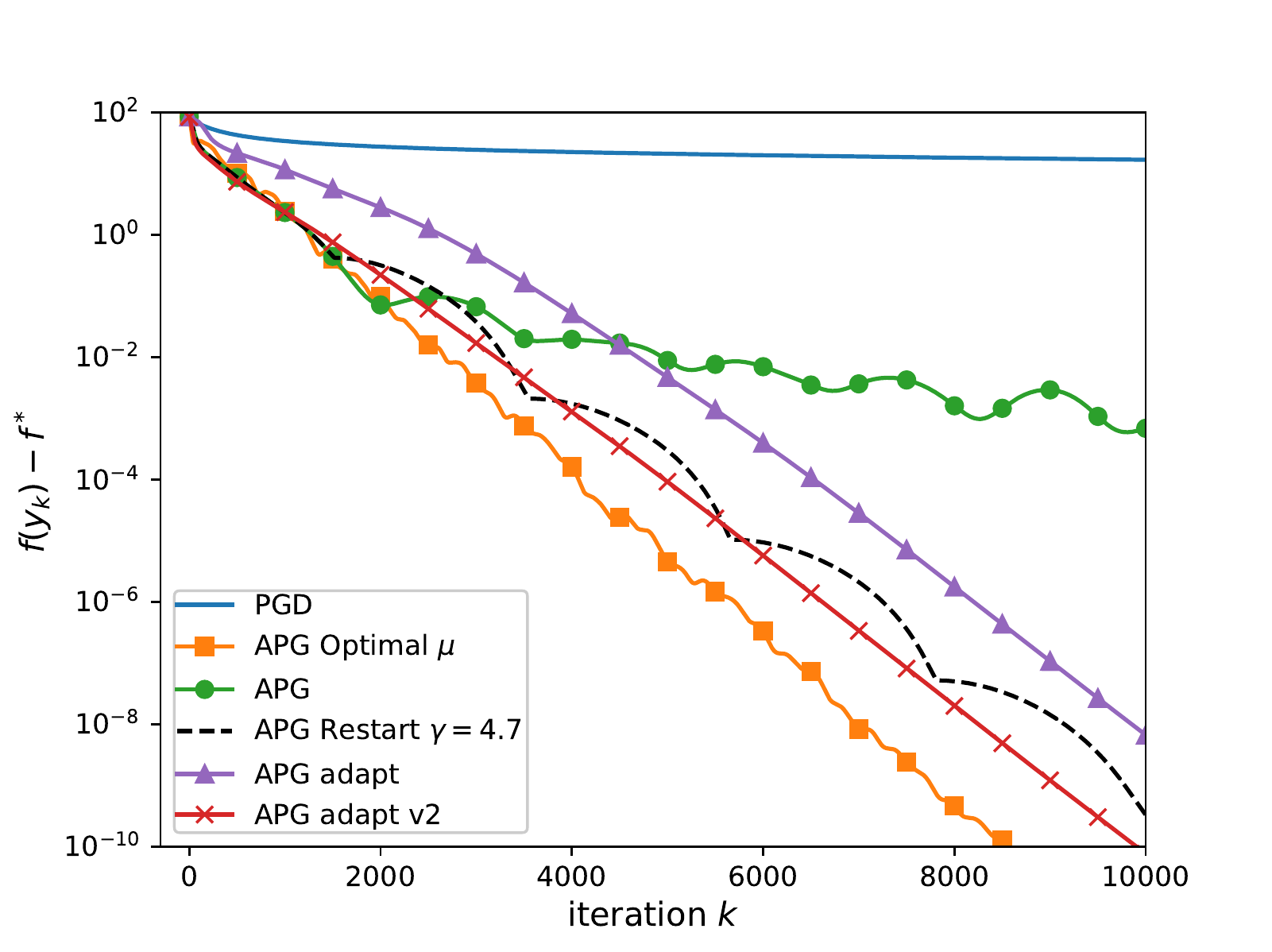} \\
       \includegraphics[scale = 0.43]{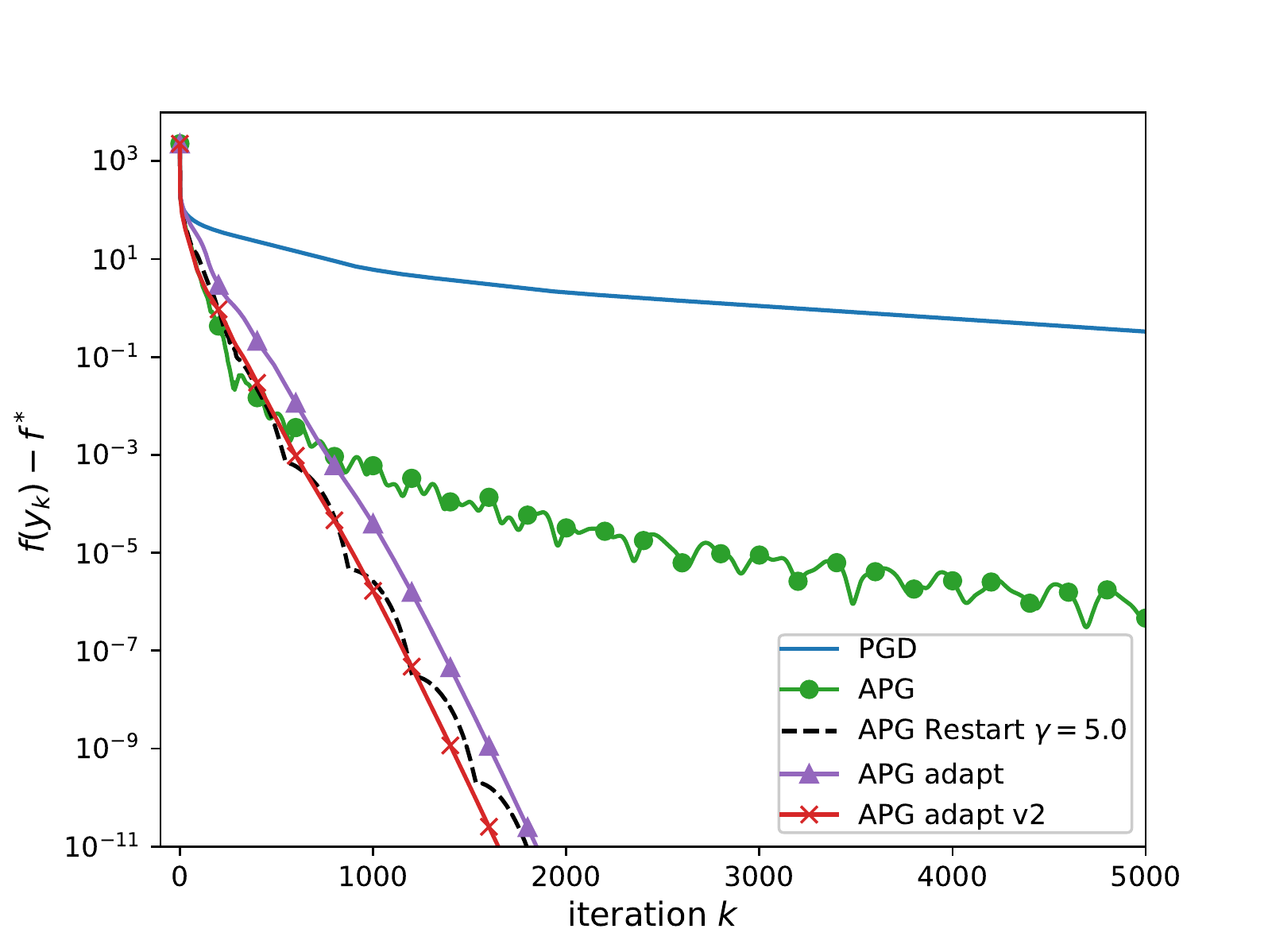} & \includegraphics[scale = 0.43]{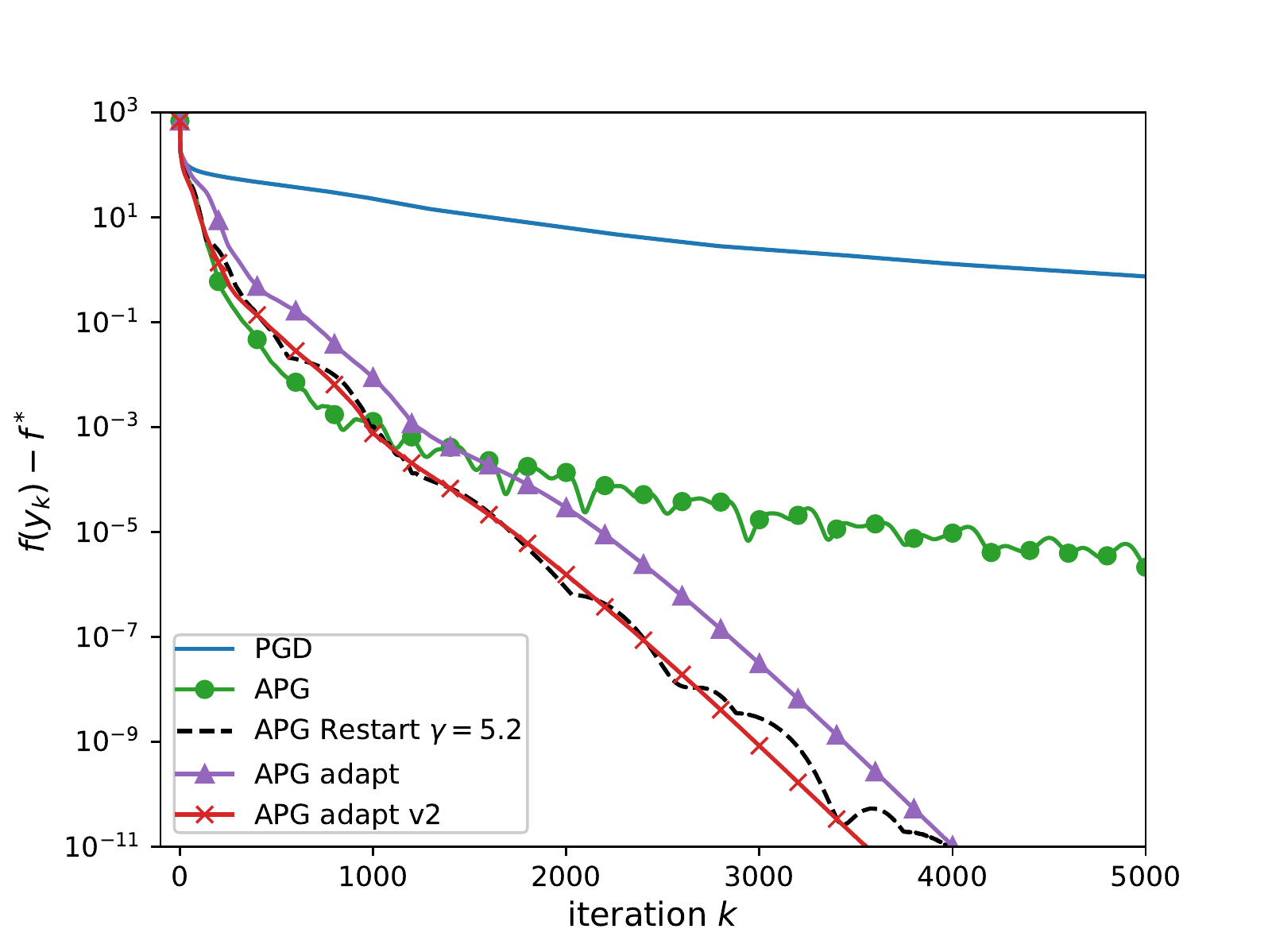}
    \end{tabular}
    \caption{Experiments on the Sonar dataset with squared loss (Top Left), regularized logistic loss with regularization parameter equal to $0.004$ (Top Right), LASSO with regularization parameter equal to 1 (Bottom Left) and dual of regularized SVM with regularization parameter equal to 1 (Bottom Right).}
    \label{fig:sonar}
\end{figure}
\newpage
In the quadratic case we dispose of a natural strong convexity parameter which is the smallest eigenvalue of the Hessian. However when the loss has a more complex structure we do not know a priori which quantity our estimates of strong convexity should be compared to. When looking at the proof of the convergence rate of Algorithm~\ref{algo:FGM-const-adapt}, the exact error term due to the fact that $\mu_k$ upper bounds $\mu$ is 
\BEQ
\frac{\mu_k}{2}\|x^*-x_{k}\|^2 - \left(f^* - f(T_L(x_{k})) - g_L(x_k)^T(x^*-x_k) -\frac{1}{2L}\|g_L(x_k)\|^2\right)
\EEQ
where $x_k$ is an iterate in Algorithm~\ref{algo:FGM-const-adapt}. We then define
\BEQ
\mu_{loc}(x) = 2\frac{f^* - f(T_L(x)) - g_L(x_k)^T(x^*-x) -\frac{1}{2L}\|g_L(x)\|^2}{\|x-x^*\|^2}
\EEQ
\begin{figure}[t]
    \centering
    \begin{tabular}{cc}
        \includegraphics[scale=0.45]{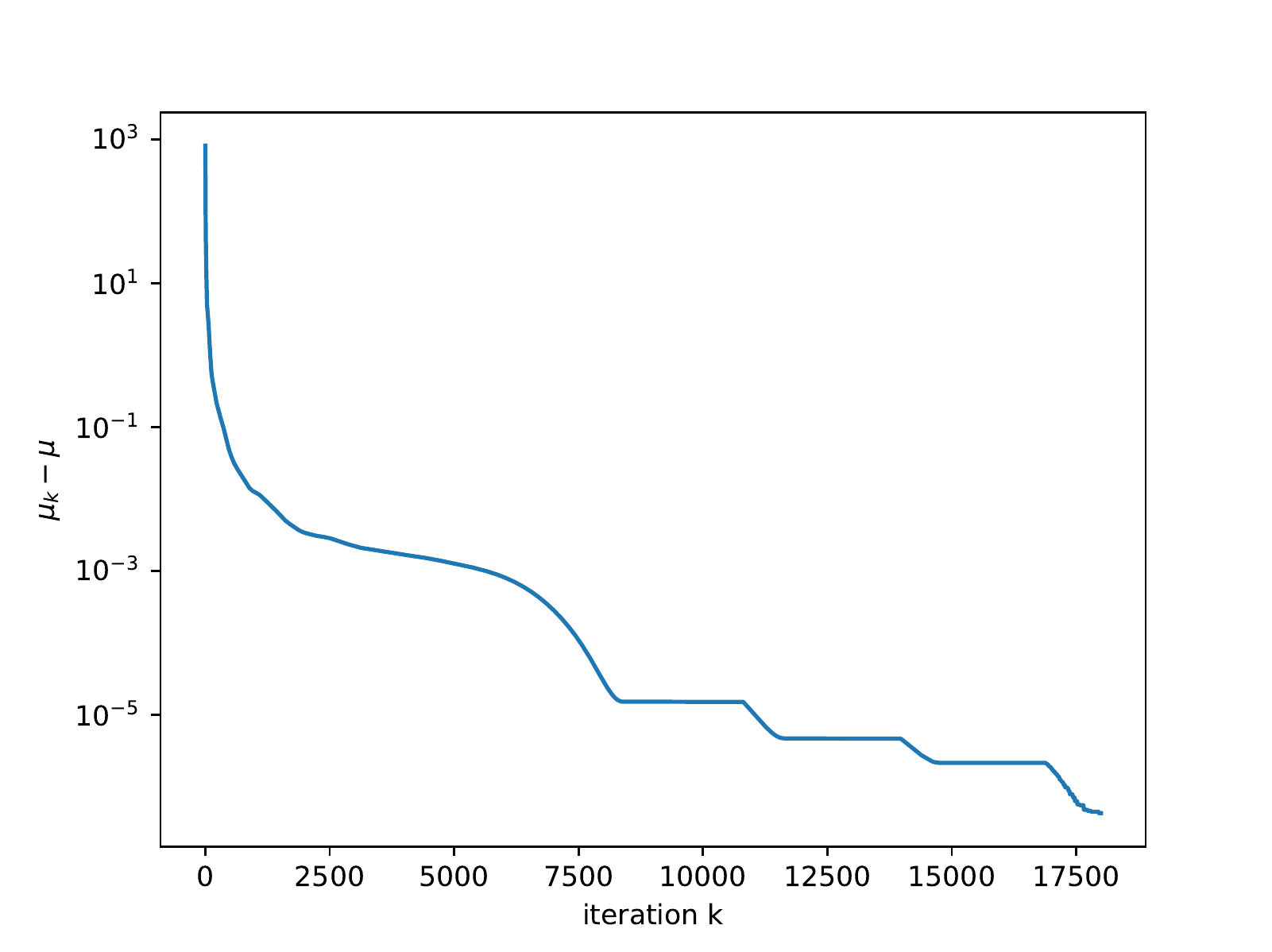}& \includegraphics[scale=0.45]{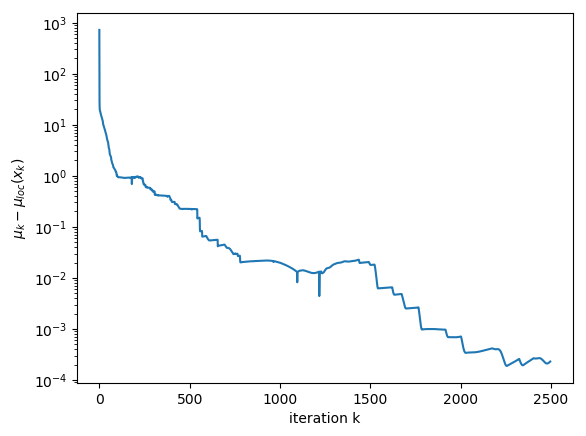}  \\
    \end{tabular}
    \caption{Left: Gap between the estimated $\mu_{k}$ and the true $\mu$ during a run of Algorithm~\ref{algo:FGM-const-adapt} on Sonar with a square loss. Right: Gap between the estimated $\mu_{k}$ and the local strong convexity parameter $\mu_{loc}(x_k)$ on Sonar with a dual SVM loss. Our estimates appear to satisfy the sublinear convergence rate needed in our robustness result. }
    \label{fig:gap}
\end{figure}
\subsection{Parameters of the losses in Figure~\ref{fig:all}}
\begin{table}[h]
    \centering
    \begin{tabular}{|c|c|c|c|}
    \hline Dataset & regularization Logit & regularization Lasso & regularization SVM\\
     \hline Musk & $\lambda\|\cdot\|^2, \lambda= 100$ & $\lambda\|\cdot\|_1,\lambda=100$ & $\frac{1}{C}\|\cdot\|^2, C = 1$ \\
     \hline Madelon & $\lambda\|\cdot\|^2, \lambda= 1000$ & $\lambda\|\cdot\|_1,\lambda=800$ & $\frac{1}{C}\|\cdot\|^2, C = 1$\\
     \hline
    \end{tabular}
    \caption{Table of the parameters used in Figure~\ref{fig:all}.}
    \label{tab:param}
\end{table}

\end{document}